\documentclass[12pt,oneside,reqno,cmex10]{amsart}
\usepackage{amsmath}
\usepackage{amsthm}
\usepackage{amssymb}

\DeclareFontFamily{OT1}{rsfs}{}
\DeclareFontShape{OT1}{rsfs}{m}{n}{ <-7> rsfs5 <7-10> rsfs7 <10-> rsfs10}{}
\DeclareMathAlphabet{\mathscr}{OT1}{rsfs}{m}{n}

\def\mysavedown#1{\edef\mysubs{\mysubs#1}}
\def\mysaveup#1{\edef\mysups{\mysups#1}}
\def\mydown#1{{\mytensor}_{\vphantom{\mysubs}#1}}
\def\myup#1{{\mytensor}^{\vphantom{\mysups}#1}}
\def\tensor#1#2{
  #1
  \def\mytensor{\vphantom{#1}}
  \def\mysubs{\relax}
  \def\mysups{\relax}
  \let\down=\mysavedown
  \let\up=\mysaveup
  #2
  \let\down=\mydown
  \let\up=\myup
  #2
  }

\newcommand{\Id}{\operatorname{Id}}

\DeclareMathOperator{\BV}{BV}
\DeclareMathOperator{\NBV}{NBV}

\newcommand{\R}{\mathbb R}
\newcommand{\B}{\mathbb B}

\renewcommand{\S}{\mathbb S}

\renewcommand{\setminus}{\smallsetminus}

\renewcommand{\to}{\rightarrow}

\renewcommand{\centerdot}{\mathbin{\text{\protect\raisebox{-.3ex}[1ex][0ex]{\Large{$\cdot$}}}}}
\newcommand{\cross}{\mathbin{\times}}

\renewcommand{\exp}{\operatorname{exp}}

\DeclareMathOperator{\End}{End}

\newcommand{\mapsinto}{\mathrel{\hookrightarrow}}
\renewcommand{\phi}{\varphi}
\renewcommand{\epsilon}{\varepsilon}

\def\crn#1#2{{\vcenter{\vbox{
        \hbox{\kern#2pt \vrule width.#2pt height#1pt
           }
          \hrule height.#2pt}}}}

\newcommand{\cadlag}{{{\it c\`adl\`ag}}}
\newcommand{\tw}{\widetilde}

\newcommand{\<}{\langle}
\renewcommand{\>}{\rangle}

\renewcommand{\hbar}{{\overline h}}

\newcommand{\pre}[2]{{{\vphantom{#2}}^{#1}}\kern-.2ex{#2}}

\theoremstyle{plain}
\newtheorem{theorem}{Theorem}[section]
\newtheorem{lemma}[theorem]{Lemma}

\newtheorem{assumption}[theorem]{Assumption}

\theoremstyle{definition}
\newtheorem{example}[theorem]{Example}

\theoremstyle{remark}
\newtheorem{remark}[theorem]{Remark}

\numberwithin{equation}{section}

\def\prt{\partial}
\def\R{{\bf R}}

\def\eps{\varepsilon}
 \def\ol{\overline}

 \def\bP{{\bf P}}
\def\bE{{\bf E}}

\def\bone{{\bf 1}}

\def\n{{\bf n}}
\def\bv{{\bf v}}
\def\bw{{\bf w}}
\def\bz{{\bf z}}
\def\prt{\partial}
\def\E{{\mathcal E}}
\def\B{{\mathcal B}}

\def\S{{\mathcal S}}
\def\T{{\mathcal T}}

\def\A{{\mathcal A}}

\def\tngt{\T}
\def\sh{\S}
\def\cD{{\mathcal D}}

\begin{document}

\baselineskip=1.3\baselineskip

\title{\bf Multiplicative functional for reflected Brownian motion
via deterministic ODE}
\author{
{\bf Krzysztof Burdzy}  \ and \ {\bf John M. Lee} }
\address{Department of Mathematics, Box 354350,
University of Washington, Seattle, WA 98195}
\thanks{Research supported in part by NSF Grants DMS-0600206 and DMS-0406060. }

\begin{abstract}
We prove that a sequence of semi-discrete approximations
converges to a multiplicative functional for reflected Brownian
motion, which intuitively represents the Lyapunov exponent for
the corresponding stochastic flow. The method of proof is based
on a study of the deterministic version of the problem and the
excursion theory.
\end{abstract}


\maketitle

\pagestyle{myheadings} \markboth{}{Multiplicative functional
for reflected Brownian motion}

\section{Introduction}\label{section:article_intro}

This article is the first part of a project devoted to path
properties of a stochastic flow of reflected Brownian motions.
We will first outline the general direction of the project and
then we will comment on the results contained in the current
article.

Consider a bounded smooth domain $D \subset\R^n$, $n\geq 2$,
and for any $x\in \ol D$, let $X^x_t$ be reflected Brownian
motion in $D$, starting from $X^x_0 = x $. Construct all
processes $X^x$ so that they are driven by the same
$n$-dimensional Brownian motion. It has been proved in
\cite{BCJ} that in some planar domains, for any $x\ne y$, the
limit $\lim_{t\to \infty} \log |X^x _t - X^y_t| /t =
\Lambda(D)$ exists a.s. Moreover, an explicit formula has been
given for the limit $\Lambda(D)$, in terms of geometric
quantities associated with $D$. Our ultimate goal is to prove
an analogous result for domains in $\R^n$ for $n\geq 3$.

The higher dimensional case is more difficult to study for
several reasons. First, we believe that the multidimensional
quantity analogous to $\Lambda(D)$ in the two dimensional case
cannot be expressed directly in terms of geometric properties
of $D$. Instead, it has to be expressed using the stationary
distribution for the normalized version of the multiplicative
functional studied in the present paper. Second,
non-commutativity of projections is a more challenging
technical problem in dimensions $n\geq 3$.

The result of \cite{BCJ} mentioned above contains an implicit
assertion about another limit, namely, in the space variable
for a fixed time. In other words, one can informally infer the
existence and value of the limit $\lim_{\eps\downarrow 0}
(X^{x+\eps \bv}_t - X^x_t)/\eps= \tw \A_t \bv$, for $\bv \in
\R^n$. The limit operator $\tw \A_t$, regarded as a function of
time, is a linear multiplicative functional of reflected
Brownian motion. Its form is considerably more complex and
interesting in dimensions $n\ge 3$ than in two dimensions.

Our overall plan is first to prove the differentiability in the
space variable stated in the last paragraph. Then we will prove
the existence and uniqueness of the stationary distribution for
the normalized version of $\tw A_t$. And then we will prove the
formula for the rate of convergence of $|X^x_t - X^y_t| $ to 0,
as $t\to \infty$.

The immediate goal of the present paper is much more modest
than the overall plan outlined above. We will deal with some
foundational issues related to the application of our main
method, excursion theory, to the convergence of semi-discrete
approximations to the multiplicative functional described
above. We will briefly review some of the existing literature
on the subject, so that we can place out own results in an
appropriate context.

The multiplicative functional $\tw A_t$ appeared in a number of
publications discussing reflected Brownian motion, starting
with \cite{A,IKpaper}, and later in \cite{IKbook,H}. None of
these publications contains the analysis of the deterministic
version of the multiplicative functional. This is what we are
going to do in Section \ref{section:determ}. In a sense, we are
trying to see whether the approach of \cite{LS} could be
applied in our case---that approach was to develop a
deterministic theory that could be applied to stochastic
processes path by path. Unfortunately, our result on
deterministic ODE's do not apply to reflected Brownian motion,
roughly speaking, for the same reason why the Riemann-Stiltjes
does not work for integrals with respect to Brownian motion.

Nevertheless, our deterministic results are not totally
disjoint from the second, probabilistic section. In fact, our
basic approach developed in Lemma \ref{lemma:finite-Y-est} is
just what we need in Section \ref{section:diff}. The main
theorem of Section \ref{section:diff} proves existence of the
multiplicative functional using semi-discrete approximations.
The result does not seem to be known in this form, although it
is obviously close to some theorems in \cite{A, IKpaper, H}.
However, the main point is not to give a new proof to a
slightly different version of a known result but to develop
estimates using excursion techniques that are analogous to
those in \cite{BCJ}, and that can be applied to study $X^x_t -
X^y_t$.

We continue with some general review of literature. The
differentiability of $X^x_t$ in the initial data was proved in
\cite{DZ} for reflected diffusions. The main difference between
our project and that in \cite{DZ} is that that paper was
concerned with diffusions in $(0,\infty)^n$, and our main goal
is to study the effect of the curvature of $\prt D$.
Deterministic transformations based on reflection were
considered, for example, in \cite{LS,DI,DR}. Synchronous
couplings of reflected Brownian motions in convex domains were
studied in \cite{CLJ1, CLJ2}, where it was proved that under
mild assumptions, $X^x_t - X^y_t$ is not 0 at any finite time.
Our estimates in Section \ref{section:diff} are so robust that
they indicate that Theorem \ref{thm:diffskor} holds for the
trace of a degenerate diffusion on $\prt D$, defined as in
\cite{CS, MO}, with the density of jumps having different
scaling properties than that for reflected Brownian motion. In
other words, the main theorem of Section \ref{section:diff} is
likely to hold in the case when the trace of the reflected
diffusion is any ``stable-like'' process on $\prt D$. We do not
present this generalization because, as far as we can tell, the
multiplicative functional $\tw \A_t$ does not represent the
limit $\lim_{\eps\downarrow 0} (X^{x+\eps \bv}_t - X^x_t)/\eps$
for flows of degenerate reflected diffusions.

\bigskip

We are grateful to Elton Hsu for very helpful advice.

\section{Deterministic differential
equation}\label{section:determ}

\subsection{Geometric Preliminaries}\label{section:intro}
Throughout this
section,
 $M$ will be a
$C^2$,
 properly embedded, orientable hypersurface (i.e., submanifold
of codimension $1$) in $\R^{n}$, endowed with a
 unit normal vector field $\n$. The properness condition means
that the inclusion map $M\mapsinto\R^{n}$ is a proper map (the
inverse image of every compact set is compact), which is
equivalent to $M$ being a a closed subset of $\R^{n}$. For any
$R>0$,
let $M_R $ denote the intersection of $M$ with the closed ball
of radius $R$ around the origin in $\R^n$, and note that $M_R$
is a compact subset of $M$.

We consider $M$ as a Riemannian manifold with the induced
metric.  We use the notation
$\left<\centerdot,\centerdot\right>$ for both the Euclidean
inner product on $\R^n$ and its restriction to $\tngt _xM$ for
any $x\in M$, and $\left|\centerdot\right|$ for the associated
norm.

For any $x\in M$, let $\pi_x\colon \R^{n}\to \tngt _x M$ denote
the orthogonal projection onto the tangent space $\tngt _x M$,
so
\begin{equation}\label{eq:pi}
\pi_x \bz  = \bz  - \<\bz ,\n(x)\>\n(x),
\end{equation}
and let $\sh (x)\colon \tngt _xM\to \tngt _xM$ denote the {\it
shape operator} (also known as the {\it Weingarten map}), which
is the symmetric linear endomorphism of $\tngt _xM$ associated
with the second fundamental form.  It is characterized by
\begin{equation}\label{eq:def-S}
\sh (x) \bv  = - \partial_\bv  \n(x), \qquad \bv \in \tngt _xM,
\end{equation}
where $\partial_\bv $ denotes the ordinary Euclidean
directional derivative in the direction of $\bv $. If $\gamma
\colon[0,T]\to M$ is a smooth curve in $M$, a {\it vector field
along $\gamma $} is a smooth map $\bv \colon [0,T]\to M$ such
that $\bv (t)\in \tngt _{\gamma (t)}M$ for each $t$. The {\it
covariant derivative of $\bv $ along $\gamma $} is given by
\begin{align*}
\cD _t\bv (t) &:= \bv '(t) - \<\bv (t), \sh (\gamma (t)) \gamma '(t)\>\n(\gamma (t)) \\
&= \bv '(t) + \<\bv (t), \partial_t (\n\circ \gamma )(t)\> \n(\gamma (t)) .
\end{align*}
The eigenvalues of $\sh (x)$ are the principal curvatures of
$M$ at $x$, and its determinant is the Gaussian curvature. We
extend $\sh (x)$ to an endomorphism of $\R^{n}$ by defining
$\sh (x)\n (x) = 0$. It is easy to check that $\sh (x)$ and
$\pi_x$ commute, by evaluating separately on $\n (x)$ and on
$\bv \in \tngt _xM$.

The following lemma expresses some elementary observations that
we will use below.  Most of these follow easily from the fact
that smooth maps satisfy uniform local Lipschitz estimates,
so we leave the proof to the reader.
For any linear map $\A\colon \R^{n}\to \R^{n}$, we let $\|\A\|$
denote the operator norm.

\begin{lemma}\label{lemma:globalK}
For any $R>0$ and $T>0$, there exists a constant $K$ depending
only on $M$, $R$, and $T$ such that the following estimates
hold for all $x,y\in M_R$, $0\le l \le T$, $t\ge 0$ and $\bz\in
\R^n$:
\begin{align}
\|\pi_x - \pi_y\| &\le K|x-y|.\label{eq:pi-est}\\
\|\sh (x)\|&\le K.\label{eq:S-norm-est}\\
\|\sh (x)-\sh (y)\| &\le K|x-y|.\\
\|e^{t \sh (x)}\| &\le e^{K t}.\label{eq:e^S-est}\\
|e^{t \sh (x)} \bz| &\ge e^{-K  t}|\bz|. \label{eq:e^S-estlow}\\
\|e^{l \sh (x)} - \Id\| &\le Kl.\label{eq:e^S-est2}\\
\|e^{l \sh (x)} - e^{l \sh (y)}\| &\le {K }l\,|x-y|.\label{eq:e^S-est3}\\
|\n (x)-\n (y)| &\le K|x-y|.\label{eq:N-Lip-est}
\end{align}
\end{lemma}

Another useful estimate is the following.

\begin{lemma}\label{lemma:pipi-pipi}
For any $R>0$, there exists a constant $C$ depending only on $M$ and $R$
such that for all
$w,x,y,z\in M_R$, the following operator-norm estimate holds:
\begin{displaymath}
\left\|\pi_{z} \circ \left(\pi_{y} - \pi_{x}\right)\circ \pi_{w}\right\|
\le C\big(\left|w-y\right|\,\left|y-z\right| +
\left|w-x\right|\,\left|x-z\right| \big).
\end{displaymath}
\end{lemma}

\begin{proof}
Using the fact that $\n $ is a unit vector field and expanding
$|\n (x)-\n (y)|^2$ in terms of inner products, we obtain
\begin{displaymath}
\<\n (x),\n (y)\> = \tfrac12 (|\n (x)|^2 + |\n (y)|^2 - |\n
(x)-\n (y)|^2) = 1 - \tfrac12 |\n (x)-\n (y)|^2.
\end{displaymath}

Suppose $w,x,y,z\in M_R$ and $\bv \in\R^n$.  If $\pi_w\bv =0$,
then the estimate holds trivially, so we may as well assume
that $\bv \in \tngt _wM$. Expanding the projections as in
\eqref{eq:pi} and using the fact that $\pi_w\bv =\bv $, we
obtain
\begin{align*}
\pi_{z}&\bigl(\pi_{y} - \pi_x\bigr)\pi_w\bv\\
&= \pi_{z}(\bv  - \<\bv ,\n (y)\>\n (y))
- \pi_{z}(\bv  - \<\bv ,\n (x)\>\n (x))\\
&= (\bv  - \<\bv ,\n (y)\>\n (y)) \\
&\quad- \bigl( \<\bv ,\n (z)\> \n (z) - \<\bv ,\n (y)\>\<\n (y),\n (z)\>\n (z)\bigr)\\
&\quad - (\bv  - \<\bv ,\n (x)\>\n (x))\\
&\quad + \bigl( \<\bv ,\n (z)\> \n (z) - \<\bv ,\n (x)\>\<\n (x),\n (z)\>\n (z)\bigr)\\
&= - \<\bv ,\n (y)\>\n (y) + \<\bv ,\n (x)\>\n (x) \\
&\quad + \<\bv ,\n (y)\>\bigl(1 - \tfrac12 |\n (y)-\n (z)|^2\bigr)\n (z)\\
&\quad - \<\bv ,\n (x)\>\bigl(1 - \tfrac12 |\n (x)-\n (z)|^2\bigr)\n (z)\\
&= -\<\bv ,\n (y)\>(\n (y)-\n (z)) + \<\bv ,\n (x)\>(\n (x)-\n (z)) \\
&\quad -\tfrac 12 \<\bv ,\n (y)\>|\n (y)-\n (z)|^2\n (z)
+ \tfrac12\<\bv ,\n (x)\>|\n (x)-\n (z)|^2\n (z).
\end{align*}
Using the fact that $\<\bv ,\n (w)\>=0$, this can be written
\begin{align*}
\pi_{z}\bigl(\pi_{y} - \pi_x\bigr)\pi_w\bv
&= -\<\bv ,\n (w) - \n (y)\>(\n (y)-\n (z)) \\
&\quad + \<\bv ,\n (w) - \n (x)\>(\n (x)-\n (z)) \\
&\quad -\tfrac 12 \<\bv ,\n (w)-\n (y)\>|\n (y)-\n (z)|^2\n (z)\\
&\quad + \tfrac12\<\bv ,\n (w) - \n (x)\>|\n (x)-\n (z)|^2\n (z).
\end{align*}
The desired estimate follows from
\eqref{eq:N-Lip-est} and the fact that
\begin{displaymath}
|\n (x)-\n (y)|^2 \le \left( |\n (x)| + |\n (y)|\right)|\n
(x)-\n (y)| \le 2K|x-y|.
\end{displaymath}
\end{proof}

\subsection{Analytic Preliminaries}

Let $T$ be a positive real number.
We let $\BV([0,T];\R)$ denote the set of functions $u\colon [0,T]\to \R$
of bounded variation,
and $\NBV([0,T];\R)\subset \BV([0,T];\R)$ the subset consisting of
functions that are right-continuous.
By convention, we will consider each
$u\in \NBV([0,T];\R)$ to be a function defined
on all of $\R$ by setting $u(t)=0$ for $t< 0$ and $u(t)=u(T)$ for $t>T$;
the extended function is still right-continuous and of bounded variation.
With this understanding,
we will follow the conventions of \cite{Folland}, and most of the
properties of $\NBV([0,T];\R)$ that we use can be found there.

It is easy to check that
$\NBV([0,T];\R)$ is closed under pointwise products and sums.
Functions in $\NBV([0,T];\R)$ have bounded images,
at most countably many
discontinuities, and well-defined left-hand limits at
each discontinuity.  In particular, they are
examples of \cadlag\ functions ({\it continue \`a droite, limites \`a gauche}).
(In fact, $\NBV([0,T];\R)$ is exactly the set of
\cadlag\ functions of bounded variation.)
For any $u\in \NBV([0,T];\R)$ and any $s\in [0,T]$,
we set
\begin{displaymath}
u(s-) = \lim_{t\nearrow s} u(t),
\end{displaymath}
and we define the {\it jump of $u$ at $s$}
to be
\begin{displaymath}
\Delta_s(u) = u(s) - u(s-).
\end{displaymath}
Note that $u(0-) = 0$ and $\Delta_0(u) = u(0)$
by our conventions.

It follows from
elementary measure theory that
for each $u\in \NBV([0,T];\R)$,
there is a unique signed Borel measure $du$ on $[0,T]$
characterized by
\begin{displaymath}
du \bigl( (a,b] \bigr) = u(b) - u(a), \quad t\in [0,T].
\end{displaymath}
Because this
measure has atoms exactly at points $t\in [0,T]$ where $u$ is discontinuous,
we have to be careful to indicate whether endpoints are included or excluded
in integrals.  For example, we have the following versions of
the fundamental theorem of calculus for $a,b\in[0,T]$:
\begin{align*}
\int_{(a,b]} du &= u(b) - u(a);&
\int_{[a,b]} du &= u(b) - u(a-);\\
\int_{(a,b)} du &= u(b-) - u(a);&
\int_{[a,b)} du &= u(b-) - u(a-).
\end{align*}
The {\it total variation} of $u$, denoted by $\|du\|$,
is given by either of two formulas:
\begin{align*}
\|du\|
&= \sup \left\{ \sum_{i=1}^k |u(x_i) - u(x_{i-1})|: 0=x_0<x_1<\dots < x_k=T\right\}\\
&= \int_{[0,T]}|du|.
\end{align*}
It follows from our conventions that $\|u\|_\infty\le \|du\|$.

For $u\in \NBV([0,T];\R)$, we will use the notation $u_-$ to denote the
function $u_-(t) = u(t-)$.  Note that $u_-$ has bounded variation, but is
left-continuous rather than right-continuous.

\begin{lemma}\label{lemma:byparts}
For any $u,v\in \NBV([0,T];\R)$ and $a,b\in[0,T]$, the following
integration by parts formula holds:
\begin{equation}\label{eq:intbyparts}
\int_{(a,b]} u\,dv + \int_{(a,b]} v_-\, du
= u(b)v(b) - u(a)v(a).
\end{equation}
\end{lemma}

\begin{proof}
This follows as in \cite[Thm.\ 3.36]{Folland} by applying
Fubini's theorem to the integral $\int_\Omega du\cross dv$,
where $\Omega$ is the triangle $\{(s,t): a< s\le t\le b\}$.
\end{proof}

\begin{lemma}\label{lemma:productrule}
The following product rules hold for $u,v\in \NBV([0,T];\R)$:
\begin{align*}
d(uv) &= u\,dv + v_-\, du \\
&= u_-\,dv + v\, du\\
&=
u\,dv + v\, du - \sum_i \Delta_{s_i}(u)\Delta_{s_i}(v)\delta_{s_i},
\end{align*}
where
$\delta_{s_i}$ is the Dirac mass at $s_i$, and
the sum is over the countably many points $s_i\in[0,T]$
at which both $u$ and $v$ are discontinous.
\end{lemma}

\begin{proof}
The first two formulas follow immediately from \eqref{eq:intbyparts}
and the definition of $d(uv)$.
For the third, we just note that the measure $(v-v_-)du$ is
supported on the set of points where $u$ and $v$ are both discontinuous,
and for each such point $s_i$,
\begin{align*}
(v(s_i)-v_-(s_i))du(\{s_i\}) &= (v(s_i) - v(s_i-)) (u(s_i) - u(s_i-)) \\
&= \Delta_{s_i}(u)\Delta_{s_i}(v).
\end{align*}
\end{proof}

We will be interested primarily in vector-valued functions. We
let $\NBV([0,T];\R^n)$ denote the set of functions $\bv \colon
[0,T]\to \R^n$ each of whose component functions is in
$\NBV([0,T];\R)$, and $\NBV([0,T];M)\subset \NBV([0,T];\R^n)$
the subset of functions taking their values in $M$. The
considerations above apply equally well to such vector-valued
functions, with obvious trivial modifications in notation. For
example, if $\bv ,\bw \in \NBV([0,T];\R^n)$, we consider $d\bv$
and $d\bw$ as $\R^n$-valued measures, and Lemma
\ref{lemma:byparts} implies that
\begin{displaymath}
\int_{(a,b]} \<\bv ,d\bw \> + \int_{(a,b]} \<\bw _-, d\bv \> =
\<\bv (b),\bw (b)\> -  \<\bv (a),\bw (a)\>.
\end{displaymath}

If $\gamma \in  \NBV([0,T];M)$, we say a function $\bv \in
\NBV([0,T];\R^n)$ is a {\it vector field along $\gamma $} if
$\bv (t)\in \tngt _{\gamma (t)}M$ for each $t\in [0,T]$. This
is equivalent to the equation $\<\bv (t),\n (\gamma (t))\>=0$
for all $t$, or more succinctly $\<\bv ,\n \circ \gamma
\>\equiv 0$. Note that the fact that $\gamma $ takes its values
in a bounded set, on which $\n $ is uniformly Lipschitz,
guarantees that $\n \circ \gamma \in \NBV([0,T];\R^n)$.

We generalize the notion of covariant derivative for $\NBV$ vector fields
by defining
\begin{displaymath}
\cD \bv  = d\bv  + \< \bv _-, d(\n \circ \gamma )\>\n \circ
\gamma .
\end{displaymath}
One motivation for this definition is provided by the following
lemma, which says that if $\bv (0)$ is tangent to $M$ and $\cD
\bv $ is tangent to $M$ on all of $[0,T]$,  then $\bv $ stays
tangent to $M$.

\begin{lemma}\label{lemma:Y-stays-tangent}
Suppose $\gamma \in  \NBV([0,T];M)$ and $\bv \in
\NBV([0,T];\R^n)$. If $\bv (0)\in \tngt _{\gamma (0)}M$ and
$\<\cD \bv ,\n \circ \gamma \> \equiv 0$, then $\bv (t)\in
\tngt _{\gamma (t)}M$ for all $t\in [0,T]$.
\end{lemma}

\begin{proof}
Using Lemma \ref{lemma:productrule}, we compute
\begin{align*}
0 &= \<\cD \bv ,\n \circ \gamma \>\\
&= \< d\bv ,\n \circ \gamma \> + \bigl\< \< \bv _-, d(\n \circ \gamma )
 \>\n \circ \gamma, \n \circ \gamma\bigr\>\\
&= d\<\bv ,\n \circ \gamma \> - \< \bv _-, d(\n \circ \gamma )\>
+ \< \bv _-, d(\n \circ \gamma ) \>
\<\n \circ \gamma, \n \circ \gamma\>\\
&= d\<\bv ,\n \circ \gamma \>.
\end{align*}
Thus if $\<\bv (0),\n (\gamma (0))\>=0$, we find by integration
that $\<\bv (t),\n (\gamma (t))\>=0$ for all $t$.
\end{proof}

\subsection{An Existence and Uniqueness Theorem}

The main purpose of this
section
 is to prove the following theorem.

\begin{theorem}\label{thm:existence-uniqueness}
Let $M\subset \R^n$ be a smooth, properly embedded
hypersurface, and let $\gamma \in \NBV([0,T];M)$. For any $\bv
_0\in \tngt _{\gamma (0)}M$, there exists a unique $\NBV$
vector field $\bv $ along $\gamma $ that is a solution to the
following (measure-valued) ODE initial-value problem:
\begin{equation}\label{eq:ODE}
\begin{aligned}
\cD \bv  &= (\sh \circ \gamma ) \bv \, dt,\\
\bv (0) &= \bv _0.
\end{aligned}
\end{equation}
\end{theorem}

Before proving the theorem, we will establish some important preliminary results.
We begin by dispensing with the uniqueness question.

\begin{lemma}
Let $\gamma \in \NBV([0,T];M)$. If $\bv ,\tw \bv \in
\NBV([0,T];\R^n)$ are both solutions to \eqref{eq:ODE} with the
same initial condition, they are equal.
\end{lemma}

\begin{proof}
Suppose $\bv $ is any solution to \eqref{eq:ODE}. Observe that
Lemma \ref{lemma:Y-stays-tangent} implies that $\bv (t)$ is
tangent to $M$ for all $t$, so $\<\bv ,\n \circ \gamma \>\equiv
0$. Let $R = \|\gamma \|_\infty$, so that $\gamma $ takes its
values in $M_R$. With $K$ chosen as in Lemma
\ref{lemma:globalK}, define $f\in \NBV([0,T];\R)$ by $f(t) =
e^{-2Kt} |\bv (t)|^2$. Then Lemma \ref{lemma:productrule}
yields
\begin{align*}
df &= e^{-2Kt} \biggl( -2K|\bv |^2 dt + 2\< \bv , d\bv \> -\sum_i \<\Delta_{s_i}\bv ,\Delta_{s_i}\bv \>\delta_{s_i}\biggr)\\
&= e^{-2Kt} \biggl( -2K|\bv |^2 dt -
 2\<\bv ,\n \circ \gamma \> \< \bv _-, d(\n \circ \gamma )\>\\
&\qquad\qquad + 2\< \bv , (\sh \circ \gamma ) \bv \>dt -\sum_i \<\Delta_{s_i}\bv ,\Delta_{s_i}\bv \>\delta_{s_i}\biggr)\\
&= e^{-2Kt} \biggl( 2\bigl( \< \bv , (\sh \circ \gamma ) \bv \> - K|\bv |^2\bigr)dt
-\sum_i |\Delta_{s_i}\bv |^2\delta_{s_i}\biggr).
\end{align*}
Since \eqref{eq:S-norm-est} shows that $\< \bv , (\sh \circ
\gamma ) \bv \> \le K|\bv |^2$, this last  expression is a
nonpositive measure on $[0,T]$.  Integrating, we conclude that
$f(t)\le f(0)$, or
\begin{displaymath}
|\bv (t)|^2\le e^{2Kt}|\bv _0|^2.
\end{displaymath}
In particular, the only solution with initial condition $\bv
_0=0$ is the zero solution. Because \eqref{eq:ODE} is linear in
$\bv $, this suffices.
\end{proof}

To prove existence, we will work first with finite
approximations. Define a {\it finite trajectory} in $M$ to be a
function $\gamma \in\NBV([0,T];M)$ that takes on only finitely
many values. This means that there exists a partition $\{0=t_0
< t_1 < \dots < t_m =T \}$ of $[0,T]$ such that $\gamma $ is
constant on $[t_{i},t_{i+1})$ for each $i$. For such a
function, $d\gamma  = \sum_{i=0}^m \Delta_{t_i}(\gamma
)\delta_{t_i}$ and $\|d\gamma \| = \sum_{i=0}^m
|\Delta_{t_i}(\gamma )|$.

Suppose $\gamma $ is a finite trajectory in $M$ and $\bv _0\in
\tngt _{\gamma (0)}M$. Let $0=t_0<\dots<t_m=T$ be  a finite
partition of $[0,T]$ including all of the discontinuities of
$\gamma $, and write $x_i = \gamma (t_i)$. Define  $\bv \colon
[0,T]\to \R^n$ by
\begin{equation}\label{eq:finite-Y}
\bv (t) =
e^{(t-t_k)\sh _{x_k}} \pi_{x_k}
e^{(t_k-t_{k-1})\sh _{x_{k-1}}} \pi_{x_{k-1}}
\cdots
e^{(t_2-t_1) \sh _{x_1}} \pi_{x_1} e^{(t_1-t_0) \sh _{x_0}} \bv _0,
\end{equation}
where $k$ is the largest index such that $t_k\le t$. Observe
that the definition of $\bv $ is unchanged if we insert more
times $t_i$ in the partition.

\begin{lemma}\label{lemma:finite-Y}
Let $\gamma \colon[0,T]\to M$ be a finite trajectory. For and
any $\bv _0\in \tngt _{\gamma (0)}M$, the map $\bv $ defined by
\eqref{eq:finite-Y} is the unique solution to \eqref{eq:ODE},
and satisfies
\begin{align}
|\bv (t)| &\le e^{Ct}|\bv _0|,\label{eq:Y(t)-est}\\
\|d\bv \| &\le C,\label{eq:dY-est}
\end{align}
where $C$ is a constant depending only on $M$, $T$, and
$\|d\gamma \|$.
\end{lemma}

\begin{proof}
An easy computation shows that
\begin{align*}
d\bv  &= (\sh \circ \gamma )\bv \, dt
+ \sum_{i=0}^m \left(
\pi_{x_i}
\bv (t_i-) - \bv (t_i-)\right) \delta_{t_i}\\
&= (\sh \circ \gamma )\bv \, dt + \sum_{i=0}^m
\bigl\<\bv (t_i-), \n (\gamma (t_i)) - \n (\gamma (t_i-))\bigr\>
\n (\gamma (t_i))
\delta_{t_i},
\end{align*}
from which it follows that $\bv $ solves \eqref{eq:ODE}.

To estimate $|\bv (t)|$, observe first that the operator norm
of each projection $\pi_{x}$ is equal to one. Let  $K$ be the
constant of Lemma \ref{lemma:globalK} for $R = \|d\gamma \|$.
Using \eqref{eq:e^S-est}, we have the following operator norm
estimate for any finite collection of points $x_1,\dots,x_j\in
M_R$ and real numbers $l_1,\dots,l_j\in [0,T]$:
\begin{equation}\label{eq:comp-est}
\|e^{l_{j} \sh _{x_j}} \circ\pi_{x_j}\circ
\cdots \circ e^{l_{1}\sh _{x_1}}\circ \pi_{x_1}\|
\le e^{Kl_{j}}\cdots e^{Kl_{1}}
= e^{K(l_{j}+\cdots+l_{1})}.
\end{equation}
Applying this to the definition of $\bv $ proves
\eqref{eq:Y(t)-est}.  Then, using \eqref{eq:Y(t)-est} and
\eqref{eq:N-Lip-est}, we estimate
\begin{align*}
\|d\bv \| &
=  \int_{[0,T]} |(\sh \circ \gamma )\bv |\,dt
+  \sum_{i=0}^m \biggl |\bigl\<\bv (t_i-),
\n (\gamma (t_i)) - \n (\gamma (t_i-))\bigr\>
\n (\gamma (t_i))
\biggr|\\
&\le \int_{[0,T]} K e^{Kt}dt +
 \sum_{i=0}^m e^{KT} K|\gamma (t_i) - \gamma (t_i-)|\\
&\le C(1+ \|d\gamma \|).
\end{align*}
\end{proof}

\begin{lemma}\label{lemma:finite-Y-est}
Suppose $\gamma $ and $\tw \gamma $ are any finite trajectories
in $M$ defined on $[0,T]$ and starting at the same point, and
$\bv $, $\tw \bv $ are the corresponding solutions to
\eqref{eq:ODE}. There is a constant $C$ depending only on $M$,
$T$, $\|\gamma \|_\infty$, and $\|\tw \gamma \|_\infty$ such
that the following estimate holds:
\begin{displaymath}
\|\bv  - \tw \bv \|_{\infty} \le C \left( 1 + \|d\gamma \| +
\|d\tw \gamma \|\right)\|\gamma -\tw \gamma \|_{\infty}|\bv
_0|.
\end{displaymath}
\end{lemma}

\begin{proof}
Lemma \ref{lemma:finite-Y} shows that $\|\bv \|_\infty$ and
$\|\tw \bv \|_\infty$ are both bounded by $C|\bv _0|$ for some
$C$ depending only on $M$, $T$, $\|\gamma \|_\infty$, and
$\|\tw \gamma \|_\infty$. Fix $t\in [0,T]$, and let
$0=t_0<\dots<t_k\le t$ denote a finite partition that includes
all of the discontinuities of $\gamma $ and $\tw \gamma $ in
$[0,t]$. We introduce the following shorthand notations:
\begin{align*}
t_{k+1}&=t, &
l_i &= t_{i+1} - t_i,\\
x_i &= \gamma (t_i), &
\tw x_i &= \tw \gamma (t_i),\\
\sh _i &= \sh (x_i),&
\tw \sh _i &= \sh (\tw x_i),\\
\pi_i &= \pi_{x_i},&
\tw\pi_i &= \pi_{\tw x_i}.
\end{align*}

Observing that $\pi_0\bv _0=\bv _0$ and $\tw \pi_{k+1}\tw \bv
(t) = \tw \bv (t)$, we can write $\bv (t)-\tw \bv (t)$ as a
telescoping sum:
\begin{equation*}
\bv (t) - \tw \bv (t) =  \sum_{i=0}^{k}   e^{l_{k} \sh _{k}}
\pi_k \cdots e^{l_{i+1} \sh _{i+1}}  \pi_{i+1} \left( e^{l_{i}
\sh _{i}} \pi_{i} - \tw \pi_{i+1} e^{l_{i} \tw \sh _{i}}
\right) \tw \pi_{i} \cdots e^{l_1 \tw \sh _{1}} \tw \pi_{1}
e^{l_0 \tw \sh _{0}} \bv _0.
\end{equation*}
By \eqref{eq:comp-est}, the compositions of operators before and after the
parentheses in the summation above are uniformly bounded
in operator norm by $e^{KT}$.
Therefore,
\begin{displaymath}
|\bv (t) - \tw \bv (t)| \le e^{2KT} \sum_{i=0}^{k}\left\|
\pi_{i+1}\circ \left( e^{l_{i}  \sh _{i}} \circ\pi_{i} - \tw
\pi_{i+1} \circ e^{l_{i} \tw \sh _{i}} \right) \circ\tw \pi_{i}
\right\|\, |\bv _0|.
\end{displaymath}

Using the fact that $\sh _{i}$ and $\pi_{i}$ commute, as do
$\tw \sh _i$ and $\tw\pi_i$, we decompose the middle factors as
follows:
\begin{align*}
\pi_{i+1}\circ
\left(
e^{l_{i}  \sh _{i}} \circ\pi_{i} -
\tw \pi_{i+1} \circ e^{l_{i} \tw \sh _{i}}
\right) \circ\tw \pi_{i}
&=
\pi_{i+1}\circ\pi_i \circ
\left(
e^{l_i \sh _i} - e^{l_i \tw \sh _i}
\right) \circ
\tw \pi_i\\
&\quad
+ \pi_{i+1} \circ
\left(
\pi_i - \tw\pi_{i+1}
\right)
\circ \tw \pi_i
\circ e^{l_i\tw \sh _i} .
\end{align*}
We will deal with each of these terms separately.

For the first term, \eqref{eq:e^S-est3} implies
\begin{displaymath}
 \left\|e^{l_i \sh _i} - e^{l_i \tw \sh _i}\right\|
\le K l_i |x_i - \tw x_i|\le K l_i \|\gamma -\tw \gamma
\|_{\infty},
\end{displaymath}
and after summing over $i$, we find that this is bounded by
$KT\|\gamma -\tw \gamma \|_{\infty}$. For the second term,
Lemma \ref{lemma:pipi-pipi} allows us to conclude that
\begin{align*}
\biggl\|
\pi_{i+1} &\circ
\left(
\pi_i - \tw\pi_{i+1}
\right)
\circ \tw \pi_i
\circ e^{l_i\tw \sh _i}
\biggr\|
\\
&\le
C \left(
\left|x_{i+1} - x_i\right|
\left|x_i - \tw x_{i}\right|
+
\left|x_{i+1} - \tw x_{i+1}\right|
\left|\tw x_{i+1} - \tw x_{i}\right|
\right)
\, \left\|e^{l_i\tw \sh _i}\right\|\\
&\le Ce^{KT} \|\gamma -\tw \gamma \|_{\infty}
\left(
\left|x_{i+1} - x_i\right|
+
\left|\tw x_{i+1} - \tw x_{i}\right|
\right).
\end{align*}
After summing, this is bounded by $Ce^{KT} \|\gamma -\tw \gamma
\|_{\infty}\left(\|d\gamma \| + \|d\tw \gamma \|\right)$. This
completes the proof.
\end{proof}

\begin{lemma}\label{lemma:unif-approx}
Let $\gamma \in\NBV([0,T];M)$ be arbitrary. For any
$\epsilon>0$, there exists a finite trajectory  $\tw \gamma
\colon [0,T]\to M$ such that $\|\gamma -\tw \gamma
\|_{\infty}<\epsilon$ and $\|d\tw \gamma \|\le \|d\gamma \|$.
\end{lemma}

\begin{proof}
Let $\epsilon$ be given.  Since $\gamma $ is \cadlag, for each
$a\in [0,T]$, there exists $\delta>0$ such that for $t\in
[0,T]$,
\begin{align}
t\in [a,a+\delta) &\implies |\gamma (t)-\gamma (a)|<\epsilon,\label{eq:jump-est-r}\\
t\in (a-\delta,a) &\implies |\gamma (t)-\gamma (a-)|<\frac{\epsilon}{2}.\label{eq:jump-est-l}
\end{align}
By compactness, we can choose finitely many points $0=a_0<a_1<\dots a_m=T$
and corresponding positive numbers $\delta_0,\dots, \delta_m$ so that
$[0,T]$ is covered by the
intervals $(a_i-\delta_i,a_i+\delta_i)$, $i=1,\dots,m$.
Because they are a cover, for each $i=1,\dots,m$ we can choose $b_i$
such that
\begin{displaymath}
b_i \in (a_{i-1},a_{i-1} + \delta_{i-1})
\cap (a_{i} - \delta_{i},a_i).
\end{displaymath}

Now define a finite trajectory $\tw \gamma \colon [0,T]\to M$
by
\begin{displaymath}
\tw \gamma (t) =
\begin{cases}
\gamma (a_{i-1}), &t\in [a_{i-1},b_i),\\
\gamma (b_i),     &t\in [b_i,a_i).
\end{cases}
\end{displaymath}
It is clear from the definition of the total variation that
$\|d\tw \gamma \|\le \|d\gamma \|$. We will show that $\|\gamma
-\tw \gamma \|_\infty<\epsilon$.

Let $t\in[0,T]$ be arbitrary.  For some $i$, either
$t\in [a_{i-1},b_i)$ or $t\in [b_i,a_i)$.
In the first case,
since $[a_{i-1},b_i) \subset [a_{i-1},a_{i-1} + \delta_{i-1})$
by construction, \eqref{eq:jump-est-r} yields
\begin{displaymath}
|\gamma (t)-\tw \gamma (t)| = |\gamma (t)-\gamma (a_{i-1})| <
\epsilon.
\end{displaymath}
On the other hand, if $t\in [b_i,a_i)\subset (a_i-\delta_i,a_i)$,
\eqref{eq:jump-est-l} yields
\begin{align*}
|\gamma (t)-\tw \gamma (t)|
&= |\gamma (t)-\gamma (b_i)|\\
&\le |\gamma (t)-\gamma (a_i-)| + |\gamma (a_i-)-\gamma (b_i)|\\
&< \frac{\epsilon}{2} + \frac{\epsilon}{2},
\end{align*}
so we reach the same conclusion.
\end{proof}

\begin{lemma}\label{lemma:sequence}
For any $\gamma \in\NBV([0,T];M)$, there exists a sequence of
finite trajectories $\gamma ^{(k)}\colon [0,T]\to M$ satisfying
$\|d\gamma ^{(k)}\|\le \|d\gamma \|$ and converging uniformly
to $\gamma $.
\end{lemma}

\begin{proof}
This is an immediate consequence of Lemma \ref{lemma:unif-approx}.
\end{proof}

Now we can prove the existence and uniqueness theorem.

\begin{proof}[Proof of Theorem \ref{thm:existence-uniqueness}]
Given $\gamma $ as in the statement of the theorem, let $\gamma
^{(k)}$ be a sequence of finite trajectories converging
uniformly to $\gamma $ as guaranteed by Lemma
\ref{lemma:sequence}. For each $k$, let $\bv ^{(k)}$ be the
solution to \eqref{eq:ODE} for $\gamma =\gamma ^{(k)}$, as
defined by \eqref{eq:finite-Y}. Then Lemma
\ref{lemma:finite-Y-est} guarantees that the sequence $\bv
^{(k)}$ is uniformly Cauchy, and hence there is a limit
function $\bv \colon [0,T]\to \R^n$ such that $\bv ^{(k)}\to
\bv $ uniformly. It is straightforward to check that $\bv \in
\NBV([0,T];\R^n)$. Moreover, since each $\bv ^{(k)}$ is tangent
to $M$ and $\bv ^{(k)}\to \bv $ uniformly, it follows that $\bv
$ is also tangent to $M$.

We need to show that $\bv $ solves  \eqref{eq:ODE} for $\gamma
$. It suffices to show for any $\bw \in \NBV([0,T];\R^n)$ that
\begin{displaymath}
\int_{[0,T]} \<\bw ,d\bv \> = - \int_{[0,T]} \<\bw ,\n \circ
\gamma \> \< \bv _-, d(\n \circ \gamma )\> + \int_{[0,T]} \<
\bw , (\sh \circ \gamma )\bv \>\,dt.
\end{displaymath}
If we write $\bw  = \bw ^\top + \bw ^\perp$, where $\bw ^\top$
is tangent to $M$ and $\bw ^\perp$ is orthogonal to $M$, this
is equivalent to the following two equations:
\begin{align}
\int_{[0,T]} \<\bw ^\perp,d\bv \>
&= - \int_{[0,T]}\<\bw ^\perp,\n \circ \gamma \> \< \bv _-, d(\n \circ \gamma )\>,\label{eq:perp-part}\\
\int_{[0,T]} \<\bw ^\top,d\bv \>
&= \int_{[0,T]} \< \bw ^\top, (\sh \circ \gamma )\bv \>\,dt.\label{eq:tan-part}
\end{align}

Because $\bw ^\perp$ is proportional to $\n $, $\bw ^\perp =
\<\bw ^\perp,\n \circ \gamma \>\n \circ \gamma $. The fact that
$\bv $ is tangent to $M$ means that $\<\n \circ \gamma ,\bv
\>\equiv 0$, from which we conclude
\begin{align*}
0 &= d\<\n \circ \gamma ,\bv \>= \< \n \circ \gamma , d\bv \>+\<\bv _-,d(\n \circ \gamma )\>.
\end{align*}
Therefore,
\begin{align*}
\<\bw ^\perp,d\bv \>
&= \bigl\< \<\bw ^\perp,\n \circ \gamma \>\n \circ \gamma , d\bv  \bigr\>\\
&= \<\bw ^\perp,\n \circ \gamma \>\< \n \circ \gamma  , d\bv \> \\
&= - \<\bw ^\perp,\n \circ \gamma \> \<\bv _-,d(\n \circ \gamma )\> ,
\end{align*}
from which \eqref{eq:perp-part} follows.

On the other hand,
from Lemma \ref{lemma:byparts} we conclude that
\begin{align*}
\int_{[0,T]} \<\bw ^\top,d\bv \>
&=
\<\bw ^\top(T),\bv (T)\>
- \int_{[0,T]} \<\bv _-,d\bw ^\top\>\\
&= \lim_{k\to \infty}\left(
\<\bw ^\top(T),\bv ^{(k)}(T)\>
- \int_{[0,T]} \<\bv _-^{(k)},d\bw ^\top\>\right)\\
&= \lim_{k\to \infty}
\int_{[0,T]} \<\bw ^\top,d\bv ^{(k)}\>\\
&= \lim_{k\to \infty}\biggl(- \int_{[0,T]}\<\bw ^\top,\n \circ \gamma ^{(k)}\> \< \bv ^{(k)}_-, d(\n \circ \gamma ^{(k)})\>\\
&\qquad \qquad+ \int_{[0,T]} \< \bw ^\top, (\sh \circ \gamma ^{(k)})\bv ^{(k)}\>\,dt\biggr).
\end{align*}
Since $\<\bw ^\top , \n \circ \gamma ^{(k)}\>$ converges
uniformly to $\<\bw ^\top ,\n \circ \gamma \> \equiv 0$, and
the measures $\< \bv ^{(k)}_-, d(\n \circ \gamma ^{(k)})\>$
have uniformly bounded total variation, the first term above
vanishes in the limit. Since both $\bv ^{(k)}$ and $\sh \circ
\gamma ^{(k)}$ converge uniformly, the last term above
converges to $\int_{[0,T]} \< \bw ^\top, (\sh \circ \gamma )\bv
\>\,dt$. This proves \eqref{eq:tan-part}.
\end{proof}

\subsection{Stability}

In this section, we wish to address the stability of the
solution to \eqref{eq:ODE} under perturbations of the
trajectory $\gamma $. For applications to probability, we will
need to consider perturbations in a weaker topology than the
uniform one.

We define a metric  $d_S$ on $\NBV([0,T];\R^n)$, called
the {\it Skhorokhod metric}, by
\begin{displaymath}
d_S(\gamma ,\tw \gamma ) = \inf_{\lambda\in\Lambda} \max\left(
\|\gamma  - \tw \gamma \circ\lambda\|_\infty,
\|\lambda-\Id\|_\infty\right),
\end{displaymath}
where $\Lambda$ is the set of increasing homeomorphisms $\lambda\colon[0,T]\to [0,T]$.
We wish to show that the solution to \eqref{eq:ODE} is continuous in the
Skorokhod metric, as long as we stay within a set of trajectories with
uniformly bounded total variation.

Because the Skorokhod metric is not homogeneous with respect to
constant multiples, it will not be possible to bound $d_S(\bv
,\tw \bv )$ directly in terms of $d_S(\gamma ,\tw \gamma )$.
For this reason, we will work instead with the {\it solution
operator}: for any $\NBV$ trajectory $\gamma \colon [0,T]\to
M$, this is the endomorphism-valued function $\A\colon [0,T]\to
\End(\R^n)$ defined by
\begin{displaymath}
\A(t)\bv _0 = \bv (t),
\end{displaymath}
where $\bv $ is the solution to \eqref{eq:ODE} with initial
value $\bv _0$, and extended to an endomorphism of $\R^n$ by
declaring $\A(t)\n _{\gamma (0)} = 0$. As before, $\|\A(t)\|$
will denote the operator norm of $\A(t)$, and we set
\begin{align*}
\|\A\|_\infty &=
\sup \{ \|\A(t)\|: t\in [0,T]\}\\
&= \sup \left\{ \frac{|\bv (t)|}{|\bv _0|}: t\in [0,T], \ \bv _0 \in \tngt _{\gamma (0)}M,\ \bv _0\ne 0\right\}.
\end{align*}
It follows easily from the results of the preceding section
that for any $\gamma \in \NBV([0,T];M)$, the solution operator
$\A$ is in $\NBV([0,T];\End(\R^n))$, and Lemma
\ref{lemma:finite-Y-est} translates immediately into the
following estimate.

\begin{lemma}\label{lemma:finite-B-est}
Suppose $\gamma $ and $\tw \gamma $ are any finite trajectories
in $M$ defined on $[0,T]$ and starting at the same point, and
$\A$, $\tw \A$ are the corresponding solution operators. There
is a constant $C$ depending only on $M$, $T$, $\|d\gamma \|$,
and $\|d\tw \gamma \|$ such that the following estimate holds:
\begin{displaymath}
\|\A- \tw \A\|_{\infty} \le C \|\gamma -\tw \gamma \|_{\infty}.
\end{displaymath}
\end{lemma}

Next we need to examine the effect of a reparametrization on the
solution associated with a finite trajectory.

\begin{lemma}\label{lemma:lambda-estimate}
Let $\gamma \colon [0,T]\to M$ be a finite trajectory, let
$\lambda\colon [0,T]\to [0,T]$ be an increasing homeomorphism,
and let $\tw \gamma  = \gamma \circ\lambda$. There is a
constant $C$ depending only on $M$, $T$, and $\|d\gamma \|$
such that the solutions $\bv $ and $\tw \bv $ to \eqref{eq:ODE}
associated to $\gamma $ and $\tw \gamma $ with the same initial
value $\bv _0$ satisfy
\begin{equation}\label{eq:dS-est-for-Y}
\|\bv  - \tw \bv \circ\lambda\|_\infty \le C\|\lambda-\Id\|_\infty|\bv _0|.
\end{equation}
\end{lemma}

\begin{proof}
As in the proof of Lemma \ref{lemma:finite-Y-est}, fix $t\in
[0,T]$ and let $0=t_0<t_1<\dots<t_k\le t$ be the points in
$[0,t]$ at which $\gamma $ is discontinuous.  Set $t_{k+1} =
t$, $x_i = \gamma (t_i)$, and $\tw t_i = \lambda(t_i)$, so that
$\gamma $ and $\tw \gamma $ are given by
\begin{align*}
\gamma (t) &= x_i \quad \text{if $t_i\le t < t_{i+1}$},\\
\tw \gamma (t) &= x_i \quad \text{if $\tw t_i\le t < \tw t_{i+1}$}.
\end{align*}
We will also use the notations
\begin{align*}
l_i &= t_{i+1}-t_i,\\
\tw l_i &= \tw t_{i+1} - \tw t_i,\\
\sh _i &= \sh (x_i),\\
\pi_i &= \pi_{x_i}.
\end{align*}

We can write $\tw \bv (\lambda(t))- \bv (t) = \tw \bv (\tw
t_{k+1}) - \bv (t_{k+1})$ as a telescoping sum:
\begin{align*}
\tw \bv (\lambda(t))- \bv (t)
&=
\left( \Id - e^{(t_{k+1} - \tw t_{k+1})\sh _k}\right)\tw \bv (\lambda(t))\\
&\quad +
\sum_{i=1}^{k} e^{l_{k} \sh _{k}} \pi_k
\cdots  e^{l_{i+1}\sh _{i+1}} \pi_{i+1}\circ \\
&\qquad
\left(
e^{(t_{i+1} - \tw t_i)\sh _{i}}  \pi_{i}  e^{\tw l_{i-1} \sh _{i-1}}
-
e^{l_i \sh _i}  \pi_i  e^{(t_i - \tw t_{i-1})\sh _{i-1}}
\right)\circ\\
&\qquad\quad   \pi_{i-1}  e^{\tw l_{i-2}\sh _{i-2}}
 \cdots  e^{\tw l_1 \sh _{1}}  \pi_{1}  e^{\tw l_0 \sh _{0}}\bv _0.
\end{align*}
By virtue of \eqref{eq:e^S-est2}, the first term is bounded by
a constant multiple of $|t_{k+1} - \tw t_{k+1}|\,|\bv _0|\le
\|\lambda - \Id\|_\infty|\bv _0|$. As before, the compositions
before and after the parentheses in the summation are uniformly
bounded in operator norm, so we need only estimate the sum
\begin{displaymath}
\sum_{i=1}^{k}\left\| e^{(t_{i+1} - \tw t_i)\sh _{i}} \circ
\pi_{i} \circ e^{\tw l_{i-1} \sh _{i-1}} - e^{l_i \sh _i} \circ
\pi_i \circ e^{(t_i - \tw t_{i-1})\sh _{i-1}} \right\|.
\end{displaymath}

Using the fact that $\pi_i$ commutes with $\sh _i$, we can
rewrite the $i$-th term in this sum as
\begin{multline*}
\left\| e^{l_i \sh _i} \circ \pi_i \circ \left( e^{(t_i - \tw
t_i)\sh _i} - e^{(t_i - \tw t_i) \sh _{i-1}}\right) e^{\tw
l_{i-1} \sh _{i-1}} \right\|
\\
\le \left\| e^{l_i \sh _i}\right\|
 \left\|
e^{(t_i - \tw t_i)\sh _i }- e^{(t_i - \tw t_i) \sh _{i-1}}
\right\| \left\|e^{\tw l_{i-1} \sh _{i-1}}\right\|.
\end{multline*}
{}From \eqref{eq:e^S-est} and \eqref{eq:e^S-est3}, this last
expression is bounded by $ C\left|t_i - \tw t_i\right|\,
\left|x_i - x_{i-1}\right|$.  Summing over $i$, we conclude
that this is bounded by $C \|\lambda-\Id\|_{\infty} \|d\gamma
\|$.
\end{proof}

\begin{lemma}\label{lemma:dS-estimate}
Suppose $\gamma , \tw \gamma \colon [0,T]\to M$ are finite
trajectories starting at the same point, and let $\A$, $\tw \A$
be the corresponding solution operators. There exists a
constant $C$ depending only on $M$, $T$, $\|d\gamma \|$, and
$\|d\tw \gamma \|$ such that
\begin{equation}\label{eq:finite-stability-est}
d_S\big(\A,\tw \A\big) \le C d_S\big(\gamma ,\tw \gamma \big).
\end{equation}
\end{lemma}

\begin{proof}
Let $\delta=d_S(\gamma ,\tw \gamma )$ and let $\epsilon>0$ be
arbitrary. By definition of the Skorokhod metric, there is an
increasing homeomorphism $\lambda\colon [0,T]\to [0,T]$ such
that $\|\gamma -\tw \gamma \circ\lambda\|_\infty\le
\delta+\epsilon$ and $\|\lambda-\Id\|_\infty\le
\delta+\epsilon$. Let $\A_1$ be the solution operator
associated with $\tw \gamma \circ\lambda$. Then
$\|\A-\A_1\|_\infty\le C(\delta+\epsilon)$ by Lemma
\ref{lemma:finite-B-est}, and $\|\tw \A -
\A_1\circ\lambda\|_\infty \le C(\delta+\epsilon)$ by Lemma
\ref{lemma:lambda-estimate}. Thus by the triangle inequality,
\begin{align*}
d_S(\A,\tw \A)
&\le
d_S(\A,\A_1) + d_S(\A_1,\tw \A)\\
&\le
\|\A-\A_1\|_\infty + \max\left( \|\tw \A - \A_1\circ\lambda\|_\infty,
\|\lambda-\Id\|_\infty\right)\\
&\le C(\delta+\epsilon) +
\max( C(\delta+\epsilon), \epsilon).
\end{align*}
Letting $\epsilon\to 0$, we obtain
\begin{displaymath}
d_S(\A,\tw \A) \le 2Cd_S(\gamma ,\tw \gamma ).
\end{displaymath}
\end{proof}

Here is our main stability result.

\begin{theorem}\label{thm:stability}
Given positive constants $R$ and $T$, there exists a constant
$C$ depending only on $M$, $R$, and $T$ such that for any
trajectories $\gamma ,\tw \gamma \in\NBV([0,T];M)$ starting at
the same point and with total variation bounded by $R$, the
corresponding solution operators $\A$ and $\tw \A$ satisfy
\begin{equation}\label{eq:stability-est}
d_S\big(\A,\tw \A\big) \le C d_S\big(\gamma ,\tw \gamma \big).
\end{equation}
\end{theorem}

\begin{proof}
By the argument in the proof of Theorem
\ref{thm:existence-uniqueness}, there exist sequences of finite
trajectories converging uniformly to $\gamma $ and $\tw \gamma
$ whose solution operators converge uniformly to $\A$ and $\tw
\A$, respectively. Thus for any $\epsilon>0$, we can choose
finite trajectories $\gamma '$ and $\tw \gamma '$, with
corresponding solution operators $\A'$ and $\tw \A'$, such that
\begin{align*}
\|\gamma '-\gamma \|_\infty&<\epsilon, &\|\tw \gamma '-\tw \gamma \|_\infty&<\epsilon,\\
\|\A'-\A\|_\infty&<\epsilon, &\|\tw \A'-\tw \A\|_\infty&<\epsilon.
\end{align*}
Then by the triangle inequality,
\begin{displaymath}
d_S(\gamma ',\tw \gamma ') \le d_S(\gamma ',\gamma ) +
d_S(\gamma ,\tw \gamma ) + d_S(\tw \gamma ,\tw \gamma ') <
d_S(\gamma ,\tw \gamma ) + 2\epsilon.
\end{displaymath}
By Lemma
\ref{lemma:dS-estimate}, we have
\begin{align*}
d_S(\A',\tw \A')&\le Cd_S(\gamma ',\tw \gamma ') \le Cd_S(\gamma ,\tw \gamma ) + 2C\epsilon.
\end{align*}
Thus by the triangle inequality once more,
\begin{align*}
d_S(\A,\tw \A)
&\le d_S(\A,\A') + d_S(\A',\tw \A') + d_S(\tw \A',\tw \A)\\
&\le \epsilon + (Cd_S(\gamma ,\tw \gamma ) + 2C\epsilon) + \epsilon.
\end{align*}
Letting $\epsilon\to 0$ completes the proof.
\end{proof}

\subsection{Base trajectories of infinite variation}

In the probabilistic context, we will have to analyze the
situation when the base trajectory $\gamma $ does not have
finite variation on finite intervals. We will now present an
example showing that some of the results proved in this section
do not extend to (all) functions $\gamma$ of infinite
variation. Hence, arguments using piecewise-constant
approximations in the probabilistic context will require some
modification of our techniques.

\begin{example}
Let $M\subset \R^2$ be the parabola $M= \{(x_1, x_2) \in \R^2:
x_2 = x_1^2\}$, with the orientation of $M$ chosen so that
$\|\sh_x\| < 1$ for all $x\in M$. Let $\gamma(t) = (0,0)$ for
$t\in [0,1]$, and for even integers $j\geq 2$, let
\begin{equation*}
\gamma_j(t) =
    \left\{
      \begin{array}{ll}
        x_j:=(j^{-1}, j^{-2}), & \hbox{for $t\in[2kj^{-3}, (2k+1)j^{-3})$, $k=0,1,\dots,j^3/2-1$,} \\
        y_j:=(-j^{-1}, j^{-2}), & \hbox{for $t\in[(2k+1)j^{-3}, (2k+2)j^{-3})$, $k=0,1,\dots,j^3/2-1$,} \\
        (j^{-1}, j^{-2}), & \hbox{for $t=1$.}
      \end{array}
    \right.
\end{equation*}
Clearly, $\gamma_j \to \gamma$ in the supremum norm on $[0,1]$,
so $d_S(\gamma_j, \gamma) \to 0$. Let $\bv_0 = (1,0)$ and let
$\bv_j(t)$ be defined as in (\ref{eq:finite-Y}), relative to
$\gamma_j$. Similarly, let $\bv(t)$ be defined by
(\ref{eq:finite-Y}) relative to $\gamma$. We have $\bv(1) =
e^{\sh_{(0,0)}} \bv_0 \ne (0,0)$.

There exists $c_1>0$ such that for all $j\geq 2$, $\bz \in
\tngt_{x_j} M$, we have $|\pi_{y_j} \bz| \leq (1- c_1 j^{-2})
|\bz|$, and similarly, $|\pi_{x_j} \bz| \leq (1- c_1 j^{-2})
|\bz|$, for $\bz \in \tngt_{y_j} M$. This implies that for some
$c_2 < 1$, $|\bv_j(1)| = | (\pi_{x_j} \circ \pi_{y_j})^{j^3/2}
\bv_0| \leq c_2 ^j$. Hence, $\lim_{j\to \infty} \bv_j(1) =
(0,0) \ne \bv(1)$. This shows that results such as Lemma
\ref{lemma:finite-Y-est} do not hold for (some) functions
$\gamma$ which do not have bounded variation.
\end{example}

\section{Multiplicative functional for reflected Brownian motion}\label{section:diff}

\bigskip

Suppose $D\subset\R^n$, $n\geq 2$, is an open connected bounded
set with $C^2$ boundary. Recall that $\n (x)$ denotes the unit
inward normal vector at $x\in\prt D$. Let $B$ be standard
$d$-dimensional Brownian motion, $x_* \in \ol D$, and consider
the following Skorokhod equation,
\begin{equation}
 X_t = x_* + B_t + \int_0^t  \n (X_s) dL_s,
 \qquad \hbox{for } t\geq 0. \label{old1.1}
\end{equation}
Here $L$ is the local time of $X$ on $\prt D$. In other words,
$L$ is a non-decreasing continuous process which does not
increase when $X$ is in $D$, i.e., $\int_0^\infty
\bone_{D}(X_t) dL_t = 0$, a.s. Equation (\ref{old1.1}) has a
unique pathwise solution $(X,L)$ such that $X_t \in \ol D$ for
all $t\geq 0$ (see \cite{LS}).

We need an extra ``cemetery point'' $\Delta$ outside $\R^n$, so
that we can send processes killed at a finite time to $\Delta$.
Excursions of $X$ from $\prt D$ will be denoted $e$ or $e_s$,
i.e., if $s< u$, $X_s,X_u\in\prt D$, and $X_t \notin \prt D$
for $t\in(s,u)$ then $e_s = \{e_s(t) = X_{t+s} ,\,
t\in[0,u-s)\}$. Let $\zeta(e_s) = u -s$ be the lifetime of
$e_s$. By convention, $e_s(t) = \Delta$ for $t\geq \zeta$, so
$e_t \equiv \Delta$ if $\inf\{s> t: X_s \in \prt D\} = t$.

Let $\sigma$ be the inverse local time, i.e., $\sigma_t =
\inf\{s \geq 0: L_s \geq t\}$, and $\E_r = \{e_s: s \leq
\sigma_r\}$. Fix some $r,\eps >0$ and let $\{e_{t_1}, e_{t_2},
\dots, e_{t_m}\}$ be the set of all excursions $e\in \E_r$ with
$|e(0) -e(\zeta-)| \geq \eps$. We assume that excursions are
labeled so that $t_k < t_{k+1}$ for all $k$ and we let $\ell_k
= L_{t_k}$ for $k=1,\dots, m$. We also let $t_0 =\inf\{t\geq 0:
X_t \in \prt D\}$, $\ell_0 =0 $, $\ell_{m+1} = r$, and $\Delta
\ell_k = \ell_{k+1} - \ell_k$. Let $x_k = e_{t_k}(\zeta-)$ for
$k=1,\dots, m$, and let $x_0=X_{t_0}$.

In this section, the boundary of $D$ will play the role of the
hypersurface $M$, i.e., $M=\prt D$. Recall that $\sh$ denotes
the shape operator and $\pi_x$ is the orthogonal projection on
the tangent space $\tngt_x \prt D$, for $x\in \prt D$. For
$\bv_0\in\R^n$, let
\begin{equation}\label{def:vr}
\bv_{r,\eps} =
\exp(\Delta\ell_m \sh(x_m)) \pi_{x_m}
\cdots
\exp(\Delta\ell_1 \sh(x_1)) \pi_{x_1}
\exp(\Delta \ell_0 \sh(x_0)) \pi_{x_0} \bv_0.
\end{equation}
Let $\A_{r,\eps}$ be a linear mapping defined by $\bv_{r,\eps}
= \A_{r,\eps} \bv_0$.

We point out that the ``multiplicative functional'' $\tw \A_t$
discussed in the the Introduction is not the same as $\A_r$
defined in this section. Intuitively speaking, $\A_r = \tw
\A_{\sigma_r}$, although we have not defined $\tw \A_t$ in a
formal way.

Suppose that $\prt D$ contains $n$ non-degenerate
$(n-1)$-dimensional spheres, such that vectors perpendicular to
these spheres are orthogonal to each other. If the trajectory
$\{X_t, 0\leq t \leq r\}$ visits the $n$ spheres and no other
part of $\prt D$, then it is easy to see that $\A_{r,\eps} = 0$
for small $\eps>0$. To avoid this uninteresting situation, we
impose the following assumption on $D$.

\begin{assumption}\label{a:A1}
For every $x\in \prt D$, the $(n-1)$-dimensional surface area
measure of $\{y\in \prt D: \<\n(y),\n(x)\> =0\}$ is zero.
\end{assumption}

\begin{theorem}\label{thm:diffskor}

Suppose that Assumption \ref{a:A1} holds. With probability 1,
for every $r>0$, the limit $\A_r := \lim_{\eps\to 0}
\A_{r,\eps}$ exists and it is a linear mapping of rank $n-1$.
For any $\bv_0$, with probability 1, $\A_{r,\eps}\bv_0\to \A_r
\bv_0$ uniformly on compact sets.
\end{theorem}

\begin{remark}
Intuitively speaking, $\A_r\bv_0$ represents the solution to
the following ODE, similar to (\ref{eq:ODE}). Let $\gamma(t) =
X(\sigma_t)$, and suppose that $\bv_0 \in \R^n$. Consider the
following ODE,
\begin{equation*}
\cD \bv  = (\sh \circ \gamma ) \bv \, dt, \qquad
\bv (0) = \pi_{x_0}\bv _0.
\end{equation*}
Then $\A_r$ is defined by $\bv(r)= \A_r\bv_0$. We cannot use
Theorem \ref{thm:existence-uniqueness} to justify this
definition of $\A_r$ because $\gamma \notin \NBV([0,r];\prt
D)$. See \cite{A}, \cite{IKpaper} or \cite{H} for various
versions of the above claim with rigorous proofs. Those papers
also contain proofs of the fact that $\A_r$ is a multiplicative
functional of reflected Brownian motion. This last claim
follows directly from our definition of $\A_r$.

\end{remark}

\begin{remark}
Recall that $B$ is standard $d$-dimensional Brownian motion and
consider the following stochastic flow,
\begin{equation}
 X_t^x = x + B_t + \int_0^t  \n (X^x_s) dL^x_s,
 \qquad \hbox{for } t\geq 0, \label{old1.1new}
\end{equation}
where $L^x$ is the local time of $X^x$ on $\prt D$. The results
in \cite{LS} are deterministic in nature, so with probability
1, for all $x\in \ol D$ simultaneously, (\ref{old1.1new}) has a
unique pathwise solution $(X^x,L^x)$. In a forthcoming paper,
we will prove that for every $r>0$, a.s., $\lim_{\eps\to0}
\sup_{\bv: |\bv| \leq 1} \left| (X^{x_0 + \eps \bv} _{\sigma_r}
- X^{x_0}_{\sigma_r})/\eps - \A_r \bv\right| =0$.

\end{remark}

The rest of this section is devoted to the proof of Theorem
\ref{thm:diffskor}. We precede the actual proof with a short
review of the excursion theory. See, e.g., \cite{M} for the
foundations of the theory in the abstract setting and \cite{Bu}
for the special case of excursions of Brownian motion. Although
\cite{Bu} does not discuss reflected Brownian motion, all
results we need from that book readily apply in the present
context.

An ``exit system'' for excursions of the reflected Brownian
motion $X$ from $\prt D$ is a pair $(L^*_t, H^x)$ consisting of
a positive continuous additive functional $L^*_t$ and a family
of ``excursion laws'' $\{H^x\}_{x\in\prt D}$. In fact, $L^*_t =
L_t$; see, e.g., \cite{BCJ}. Recall that $\Delta$ denotes the
``cemetery'' point outside $\R^n$ and let ${\mathcal C}$ be the
space of all functions $f:[0,\infty) \to \R^n\cup\{\Delta\}$
which are continuous and take values in $\R^n$ on some interval
$[0,\zeta)$, and are equal to $\Delta$ on $[\zeta,\infty)$. For
$x\in \prt D$, the excursion law $H^x$ is a $\sigma$-finite
(positive) measure on $\mathcal C$, such that the canonical
process is strong Markov on $(t_0,\infty)$, for every $t_0>0$,
with transition probabilities of Brownian motion killed upon
hitting $\prt D$. Moreover, $H^x$ gives zero mass to paths
which do not start from $x$. We will be concerned only with
``standard'' excursion laws; see Definition 3.2 of \cite{Bu}.
For every $x\in \prt D$ there exists a unique standard
excursion law $H^x$ in $D$, up to a multiplicative constant.

Recall that excursions of $X$ from $\prt D$ are denoted $e$ or
$e_s$, i.e., if $s< u$, $X_s,X_u\in\prt D$, and $X_t \notin
\prt D$ for $t\in(s,u)$ then $e_s = \{e_s(t) = X_{t+s} ,\,
t\in[0,u-s)\}$ and $\zeta(e_s) = u -s$. By convention, $e_s(t)
= \Delta$ for $t\geq \zeta$, so $e_t \equiv \Delta$ if
$\inf\{s> t: X_s \in \prt D\} = t$.

Recall that $\sigma_t = \inf\{s\geq 0: L_s \geq t\}$ and let
$I$ be the set of left endpoints of all connected components of
$(0, \infty)\setminus \{t\geq 0: X_t\in \partial D\}$. The
following is a special case of the exit system formula of
\cite{M},
\begin{equation}
\bE \left[ \sum_{t\in I} V_t \cdot f ( e_t)
\right] = \bE \int_0^\infty V_{\sigma_s}
 H^{X(\sigma_s)}(f) ds = \bE \int_0^\infty V_t H^{X_t}(f) dL_t,
 \label{old4.1}
\end{equation}
where $V_t$ is a predictable process and $f:\, {\mathcal
C}\to[0,\infty)$ is a universally measurable function which
vanishes on excursions $e_t$ identically equal to $\Delta$.
Here and elsewhere $H^x(f) = \int_{\mathcal C} f dH^x$.

The normalization of the exit system is somewhat arbitrary, for
example, if $(L_t, H^x)$ is an exit system and $c\in(0,\infty)$
is a constant then $(cL_t, (1/c)H^x)$ is also an exit system.
Let $\bP^y_D$ denote the distribution of Brownian motion
starting from $y$ and killed upon exiting $D$. Theorem 7.2 of
\cite{Bu} shows how to choose a ``canonical'' exit system; that
theorem is stated for the usual planar Brownian motion but it
is easy to check that both the statement and the proof apply to
the reflected Brownian motion in $\R^n$. According to that
result, we can take $L^*_t$ to be the continuous additive
functional whose Revuz measure is a constant multiple of the
surface area measure on $\prt D$ and $H^x$'s to be standard
excursion laws normalized so that
\begin{equation}
H^x (A) =
\lim_{\delta\downarrow 0} \frac1{ \delta} \,\bP_D^{x +
 \delta\n(x)} (A),\label{old4.2}
\end{equation}
for any event $A$ in a $\sigma$-field generated by the process
on an interval $[t_0,\infty)$, for any $t_0>0$. The Revuz
measure of $L$ is the measure $dx/(2|D|)$ on $\prt D$, i.e., if
the initial distribution of $X$ is the uniform probability
measure $\mu$ in $D$ then $\bE^\mu \int_0^1 \bone_A (X_s) dL_s
= \int_A dx/(2|D|)$ for any Borel set $A\subset \prt D$, see
Example 5.2.2 of \cite{FOT}. It has been shown in \cite{BCJ}
that $(L^*_t, H^x)=(L_t, H^x)$ is an exit system for $X$ in
$D$, assuming the above normalization.

\begin{proof}[Proof of Theorem \ref{thm:diffskor}]

The overall structure of our argument will be similar to that in the proof of Lemma \ref{lemma:finite-Y-est}.

We will first consider the case $r=1$. Let $\eps_j = 2^{-j}$,
for $j\geq 1$. Fix some $j$ for now and suppose that $\eps'
\in[\eps_{j+1}, \eps_j)$. Let
\begin{align*}
\left\{e_{t^j_1}, e_{t^j_2}, \dots,
e_{t^j_{m_j}}\right\}
&=\{e\in \E_1: |e(0) - e(\zeta-)| \geq \eps_j\}, \\
\left\{e_{t'_1}, e_{t'_2}, \dots, e_{t'_{m'}}\right\}
&=\{e\in \E_1: |e(0) - e(\zeta-)| \geq \eps'\} .
\end{align*}
We label the excursions so that $t^j_k < t^j_{k+1}$ for all $k$
and we let $\ell^j_k = L_{t^j_k}$ for $k=1,\dots, m_j$.
Similarly, $t'_k < t'_{k+1}$ for all $k$ and $\ell'_k =
L_{t'_k}$ for $k=1,\dots, m'$. We also let $t^j_0 = t'_0 =
\inf\{t\geq 0: X_t \in \prt D\}$, $\ell^j_0 = \ell'_0 =0 $,
$\ell^j_{m_j+1} = \ell'_{m'+1}= 1$, $\Delta \ell^j_k =
\ell^j_{k+1} - \ell^j_k$, and $\Delta \ell'_k = \ell'_{k+1} -
\ell'_k$. Let $x^j_k = e_{t^j_k}(\zeta-)$ for $k=1,\dots, m_j$,
and $x'_k = e_{t'_k}(\zeta-)$ for $k=1,\dots, m'$. Let
$x^j_0=X_{t^j_0}$, and $x'_0=X_{t'_0}$.

Let $\gamma^j(s) = x^j_k$ for $s\in[\ell^j_k, \ell^j_{k+1})$
and $k=0,1,\dots, m_j$, and $\gamma^j(1) =
\gamma^j(\ell^j_{m_j})$. Let $\gamma'(s) = x'_k$ for
$s\in[\ell'_k, \ell'_{k+1})$ and $k=0,1,\dots, m'$, and
$\gamma'(1) = \gamma'(\ell'_{m'})$.

For $\bv_0\in\R^n$, let
\begin{align*}
\bv^j &=
\exp(\Delta\ell^j_{m_j} \sh(x^j_{m_j})) \pi_{x^j_{m_j}}
\cdots
\exp(\Delta\ell^j_1 \sh(x^j_1)) \pi_{x^j_1}
\exp(\Delta \ell^j_0 \sh(x^j_0)) \pi_{x^j_0} \bv_0, \\
\bv' &=
\exp(\Delta\ell'_{m'} \sh(x'_{m'})) \pi_{x'_{m'}}
\cdots
\exp(\Delta\ell'_1 \sh(x'_1)) \pi_{x'_1}
\exp(\Delta \ell'_0 \sh(x'_0)) \pi_{x'_0} \bv_0.
\end{align*}

Let $0=\ell_0<\dots<\ell_{m+1} = 1$ denote the ordered set of
all $\ell^j_k$'s, $0\leq k \leq m_j+1$, and $\ell'_k$'s, $0\leq
k \leq m'+1$. In the definition of $\ell_k$'s, we followed the
proof of Lemma \ref{lemma:finite-Y-est} word by word, for
conceptual consistency, although the set of $\ell_k$'s is the
same as the set of $\ell'_k$'s.

We introduce the following shorthand notations, $\Delta_i =
\ell_{i+1} - \ell_i$,
\begin{align*}
x_i &= \gamma^j(\ell_i), &
\tw x_i &=  \gamma' (\ell_i),\\
\sh _i &= \sh (x_i),&
\tw \sh _i &= \sh (\tw x_i),\\
\pi_i &= \pi_{x_i},&
\tw\pi_i &= \pi_{\tw x_i}.
\end{align*}

Observing that $\pi_0\tw \pi_{0} \bv _0=\tw \pi_{0} \bv _0$ and
$\tw \pi_{m+1} \bv' = \bv'$, we can write $\bv^j - \bv'$ as a
telescoping sum:
\begin{equation*}
\bv^j - \bv ' =  \sum_{i=0}^{m}   e^{\Delta_{m} \sh _{m}}
\pi_m \cdots e^{\Delta_{i+1} \sh _{i+1}}  \pi_{i+1} \left( e^{\Delta_{i}
\sh _{i}} \pi_{i} - \tw \pi_{i+1} e^{\Delta_{i} \tw \sh _{i}}
\right) \tw \pi_{i} \cdots e^{\Delta_1 \tw \sh _{1}} \tw \pi_{1}
e^{\Delta_0 \tw \sh _{0}} \tw \pi_{0} \bv _0.
\end{equation*}
By \eqref{eq:comp-est}, the compositions of operators before
and after the parentheses in the summation above are uniformly
bounded in operator norm by a constant. Therefore, for some
$c_1$ depending only on $D$,
\begin{equation}\label{eq:vminv}
|\bv^j - \bv'| \le c_1 \sum_{i=0}^{m}\left\| \pi_{i+1}\circ
\left( e^{\Delta_{i}  \sh _{i}} \circ\pi_{i} - \tw \pi_{i+1}
\circ e^{\Delta_{i} \tw \sh _{i}} \right) \circ\tw \pi_{i}
\right\|\, |\bv _0|.
\end{equation}

Using the fact that $\sh _{i}$ and $\pi_{i}$ commute, as do
$\tw \sh _i$ and $\tw\pi_i$, we decompose the middle factors as
follows:
\begin{align}\label{eq:decom}
\pi_{i+1}\circ
\left(
e^{\Delta_{i}  \sh _{i}} \circ\pi_{i} -
\tw \pi_{i+1} \circ e^{\Delta_{i} \tw \sh _{i}}
\right) \circ\tw \pi_{i}
&=
\pi_{i+1}\circ\pi_i \circ
\left(
e^{\Delta_i \sh _i} - e^{\Delta_i \tw \sh _i}
\right) \circ
\tw \pi_i\\
&\quad
+ \pi_{i+1} \circ
\left(
\pi_i - \tw\pi_{i+1}
\right)
\circ \tw \pi_i
\circ e^{\Delta_i\tw \sh _i} . \nonumber
\end{align}
We will deal with each of these terms separately.

For the first term, we have by (\ref{eq:e^S-est3}),
\begin{equation}\label{eq:1stterm}
 \left\|\pi_{i+1}\circ\pi_i \circ \left( e^{\Delta_i \sh _i} - e^{\Delta_i
 \tw \sh _i} \right) \circ \tw \pi_i \right\|
 \leq
 \left\| e^{\Delta_i \sh _i} - e^{\Delta_i \tw \sh _i}
 \right\|
 \leq
 c_2 \Delta_i |x_i - \tw x_i|.
\end{equation}

For the second term, Lemma \ref{lemma:pipi-pipi} and
(\ref{eq:e^S-est}) allow us to conclude that
\begin{align}
\biggl\|
\pi_{i+1} \circ
\left(
\pi_i - \tw\pi_{i+1}
\right)
\circ \tw \pi_i
\circ e^{\Delta_i\tw \sh _i}
\biggr\|
&
\le
c_3 \left(
\left|x_{i+1} - x_i\right|
\left|x_i - \tw x_{i}\right|
+
\left|x_{i+1} - \tw x_{i+1}\right|
\left|\tw x_{i+1} - \tw x_{i}\right|
\right)
\, \left\|e^{\Delta_i\tw \sh _i}\right\| \nonumber \\
&\le
c_4 \left(
\left|x_{i+1} - x_i\right|
\left|x_i - \tw x_{i}\right|
+
\left|x_{i+1} - \tw x_{i+1}\right|
\left|\tw x_{i+1} - \tw x_{i}\right|
\right). \label{eq:2ndterm}
\end{align}

We will now estimate $\bE \sup_{0\leq i \leq m} |x_i - \tw
x_i|$. Suppose that $x_i \ne \tw x_i$ for some $i$. Then there
exist $k_1$ and $k_2$ such that
 $\ell^j_{k_1} < \ell'_{k_2} < \ell^j_{k_1+1}$,
$x_i = x^j_{k_1}$, and $\tw x_i = x'_{k_2}$. Hence,
\begin{equation}\label{eq:diff1}
\{|x_i - \tw x_i| > a \}
\subset \bigcup_k
\left\{\sup_{t^j_{k} + \zeta(e^j_{k}) < t < t^j_{k+1}, X_t \in \prt D}
|x^j_{k} - X_t| >a \right\}.
\end{equation}
Intuitively speaking, the last condition means that the process
$X$ deviates by more than $a$ units from $x^j_{k_1}$ (the right
endpoint of an excursion $e_{t^j_{k_1}}$), when $X$ is on the
boundary of $D$, at some time between the lifetime of this
excursion and the start of the next excursion in this family,
$e_{t^j_{k_1+1}}$.

Since $\prt D$ is $C^2$, standards estimates (see, e.g.,
\cite{Bu}) show that for some $a_0,c_5 >0$, all $x\in \prt D$
and $a\in(0,a_0)$,
\begin{equation}\label{eq:H1}
1/(c_5 a) \leq H^x\left(|e(\zeta-) - x| > a\right) \leq c_5/ a.
\end{equation}
It follows from this and (\ref{old4.1}) that there exists $c_6$
so large that for any stopping time $T$ and $a\in(0,a_0)$,
\begin{equation}\label{eq:H2}
\bP\left( \exists e_s: |e_s(\zeta-) - e_s(0)| > a,
s\in (T, \sigma(L_T + c_6a)) \right) \geq 3/4.
\end{equation}

Let $\tau_{\B(x,a)}$ be the exit time of $X$ from the ball
$\B(x,a)$ in $\R^n$ with center $x$ and radius $a$. Routine
estimates show that for some $c_7, a_1>0$, and all $a\in (0,
a_1)$ and $x\in \prt D$,
\begin{equation}\label{eq:local1}
\bP^x( L(\tau_{\B(x,c_7 a)}) > c_6 a) > 3/4.
\end{equation}

Let $T^j_{k,0} = t^j_k$, and
$$
T^j_{k,i+1} = \inf\{t \geq T^j_{k,i}: X(t) \in \prt D,
|X(t) - X(T^j_{k,i})| \geq c_7 \eps_j\},
\qquad i \geq 0.
$$
According to (\ref{eq:local1}), the amount of local time
generated on $(T^j_{k,0}, T^j_{k,1})$ will be greater than $c_6
\eps_j$ with probability greater than $3/4$. This and
(\ref{eq:H2}) imply that there exists an excursion $e_s$ with
$|e_s(\zeta-) - e_s(0)| > \eps_j$ and $s\in (T^j_{k,0},
T^j_{k,1})$, with probability greater than $1/2$. By the strong
Markov property, if there does not exist an excursion $e_s$
with $|e_s(\zeta-) - e_s(0)| > \eps_j$ and $s\in (T^j_{k,0},
T^j_{k,i})$ then there exists an excursion $e_s$ with
$|e_s(\zeta-) - e_s(0)| > \eps_j$ and $s\in (T^j_{k,i},
T^j_{k,i+1})$, with probability greater than $1/2$. Let $M^j_k$
be the smallest $i$ with the property that there exists an
excursion $e_s$ with $|e_s(\zeta-) - e_s(0)| > \eps_j$ and
$s\in (T^j_{k,i}, T^j_{k,i+1})$. We see that $M^j_k$ is
majorized by a geometric random variable $\tw M^j_k$ with mean
2. Note that
\begin{equation*}
|X(T^j_{k,i+1}) - X(T^j_{k,i})| \leq (c_7+1) \eps_j= c_8 \eps_j,
\end{equation*}
for $i <  M^j_k$. Therefore,
\begin{equation}\label{eq:xM}
    \sup_{t^j_{k} + \zeta(e^j_{k}) < t < t^j_{k+1}, X_t \in \prt D} |x^j_{k} -
X_t| \leq c_8 M^j_k\eps_j.
\end{equation}
It is easy to see, using the strong Markov property at the
stopping times $t^j_k$, that we can assume that all $\{\tw
M^j_k, k\geq 0\}$ are independent.

Consider an arbitrary $\beta_1< -1$ and let $n_j =
\eps_j^{\beta_1}$. For some $c_9>0$, not depending on $j$,
\begin{equation}\label{eq:geom}
\bP\left(\max_{1\leq k \leq n_j} c_8 \tw M^j_k \eps _j \geq c_8 i \eps_j\right)
= 1 - (1 - (1/2)^i)^{n_j}
\leq
\begin{cases}
1& \text{if $i \leq \beta_1 j$,} \\
c_9 n_j (1/2)^i& \text{ if $i > \beta_1 j$}.
\end{cases}
\end{equation}

Let $\rho_0$ be the diameter of $D$ and $j_1 $ be the largest
integer smaller than $\log \rho_0$. By (\ref{eq:geom}), for any
$\beta_2 < 1$, some $c_{12}<\infty$, and all $j\geq j_1$,
\begin{align}
\bE\left (\max_{1\leq k \leq n_j} c_8 M^j_k \eps _j \right)
&\leq
\bE\left (\max_{1\leq k \leq n_j} c_8 \tw M^j_k \eps _j \right)
\nonumber \\
&\leq
\sum_{i \leq \beta_1 j} c_8 i \eps_j + \sum_{i >  \beta_1 j}
c_8 i \eps_j
c_9 n_j (1/2)^i \nonumber \\
&\leq c_{10} \eps_j (\log \eps_j)^2
+ c_{11} \eps_j |\log \eps_j|
\leq c_{12} \eps_j^{\beta_2} .\label{eq:maxM}
\end{align}

Let $N_\eps$ be the number of excursions $e_s$ with $s \leq
\sigma_1$ and $|e_s(0) - e_s(\zeta-)| \geq \eps$. For $\eps =
\eps_j$, $N_\eps = m_j$. Then (\ref{old4.1}) and (\ref{eq:H1})
imply that $N_\eps$ is stochastically majorized by a Poisson
random variable $\tw N_\eps$ with mean $ c_{13} /\eps$, where
$c_{13}<\infty$ does not depend on $\eps>0$. We have $\bE
\exp(\tw N_\eps) = \exp(c_{13} \eps^{-1} (e-1))$, so for any
$a>0$,
\begin{equation*}
\bP(N_\eps \geq a)
\leq \bP(\tw N_\eps \geq a) =
\bP(\exp(\tw N_\eps) \geq \exp( a)) \leq
\exp(c_{14} \eps^{-1} - a).
\end{equation*}
Standard calculations yield the following estimates. For any
$\beta_3<-1$, $\beta_4 < 0$, $\delta_1>0$, some $ \delta_2
\in(0,\delta_1)$, and all $\delta_3,\delta_4 \in (0,\delta_2)$,
\begin{equation}\label{eq:LDP1}
\bP(N_{\delta_3} \geq  \delta_3^{\beta_3})
\leq \delta_3^2,
\end{equation}
and
\begin{equation}\label{eq:LDP2}
\sup_{\delta_4\leq \delta \leq \delta_1 }
 \bE\left(N_\delta
 \bone_{\left\{N_\delta\geq  \delta^{\beta_3} \delta_4^{\beta_4}\right\}}\right)
\leq \delta_4^2.
\end{equation}

It follows from (\ref{eq:xM}), (\ref{eq:maxM}) and
(\ref{eq:LDP1}) that, for any $\beta_2 < 1$, some $c_{16}$, and
$j\geq j_1$,
\begin{align}
\bE&\left( \max_{0 \leq k \leq m_j}
\sup_{t^j_{k} + \zeta(e^j_{k}) < t < t^j_{k+1}, X_t \in \prt D}
|x^j_{k} - X_t| \right) \nonumber \\
&\leq
\bE\left( \max_{0 \leq k \leq n_j}
\sup_{t^j_{k} + \zeta(e^j_{k}) < t < t^j_{k+1}, X_t \in \prt D}
|x^j_{k} - X_t| \right)
+ \rho_0 \bP(m_j \geq n_j) \nonumber \\
&\leq
\bE\left( \max_{0 \leq k \leq n_j}
c_8M^j_k\eps_j
\right)
+ c_{15} \eps_j^2 \nonumber \\
&\leq c_{12} \eps_j^{\beta_2} + c_{15} \eps_j^2 \leq c_{16}
\eps_j^{\beta_2}. \label{eq:X-M}
\end{align}
Note that $\sum_{i=0}^m  \Delta_i = 1$. This, (\ref{eq:X-M}),
(\ref{eq:1stterm}) and (\ref{eq:diff1}) imply that,
\begin{align}\label{eq:1stterm1}
\bE &\left( \sup_{\eps_{j+1} \leq \eps' < \eps_j}
 \sum_{i=0}^m
 \left\|\pi_{i+1}\circ\pi_i \circ \left( e^{\Delta_i \sh _i} - e^{\Delta_i
 \tw \sh _i} \right) \circ \tw \pi_i \right\| \right)
 \leq
 \bE\left( \sup_{\eps_{j+1} \leq \eps' < \eps_j}
 \sum_{i=0}^m c_2 \Delta_i |x_i - \tw x_i| \right) \\
 & \leq
 \bE\left( \sup_{\eps_{j+1} \leq \eps' < \eps_j}
 \max_{0 \leq i \leq m}
|x_i - \tw x_i| \sum_{i=0}^m c_2 \Delta_i  \right)
=
 c_2 \bE\left( \sup_{\eps_{j+1} \leq \eps' < \eps_j}
  \max_{0 \leq i \leq m}
|x_i - \tw x_i|  \right) \nonumber \\
 & \leq c_2
 \bE\left( \max_{0 \leq k \leq m_j}
\sup_{t^j_{k} + \zeta(e^j_{k}) < t < t^j_{k+1}, X_t \in \prt D}
|x^j_{k} - X_t| \right) \leq c_{16}
\eps_j^{\beta_2}. \nonumber
\end{align}

We will now estimate the right hand side of (\ref{eq:2ndterm}).
We start with an observation similar to (\ref{eq:diff1}).
Suppose that $x_i \ne  x_{i+1}$ for some $i$. Then there exists
$k_1$ such that $x_i = x^j_{k_1}$, and $ x_{i+1} =
x^{j}_{k_1+1}$. Note that $k_1$'s corresponding to distinct
$i$'s are distinct. Hence,
\begin{align}\label{eq:diam}
&\{|x_i -  x_{i+1}| > a\} \\
& \nonumber \subset
\bigcup_k \left\{ \sup_{t^j_{k} +
\zeta(e^j_{k}) < t < t^j_{k+1}, X_t \in \prt D} |x^j_{k} - X_t| >a /2
\right\}
\cup \left\{|e_{t^{j}_{k+1}}(0)- e_{t^{j}_{k+1}}(\zeta-)|
> a/2\right\} \\
& \nonumber \subset
\bigcup_k \left\{ \sup_{t^j_{k} +
\zeta(e^j_{k}) < t < t^j_{k+1}, X_t \in \prt D} |x^j_{k} - X_t| >a /2
\right\}
\cup \bigcup_k \left\{|e_{t^{j+1}_{k}}(0)- e_{t^{j+1}_{k}}(\zeta-)|
> a/2\right\}.
\end{align}
Similarly, suppose that $\tw x_i \ne  \tw x_{i+1}$ for some
$i$. Then there exists $k_2$ such that $\tw x_i = x'_{k_2}$,
and $ \tw x_{i+1} = x'_{k_2+1}$. Again, $k_2$'s corresponding
to distinct $i$'s are distinct. Hence,
\begin{equation*}
\{|\tw x_i - \tw x_{i+1}| > a \}
\subset \bigcup_k \left\{
\sup_{t'_{k} + \zeta(e'_{k}) < t < t'_{k+1}, X_t \in \prt D}
|x'_{k} - X_t| >a/2 \right\}
\cup
\left\{|e_{t'_{k+1}}(0), e_{t'_{k+1}}(\zeta-)|
> a/2
\right\}.
\end{equation*}
Since $\eps_{j+1} \leq \eps' < \eps_j$, this implies that,
\begin{align}\label{eq:diam1}
&\{|\tw x_i - \tw x_{i+1}| > a \} \\
&\subset
\bigcup_{0 \leq k \leq m_j}
 \left\{
\sup_{t^j_{k} + \zeta(e^j_{k}) < t < t^j_{k+1}, X_t \in \prt D}
|x^j_{k} - X_t| >a /2 \right\}
\cup \bigcup_{0 \leq k \leq m_{j+1}}
\left\{|e_{t^{j+1}_{k}}(0)- e_{t^{j+1}_{k}}(\zeta-)|
> a/2\right\}. \nonumber
\end{align}

It follows from (\ref{eq:diff1}), (\ref{eq:diam}) and
(\ref{eq:diam1}) that
\begin{align}
&\sup_{\eps_{j+1} \leq \eps' < \eps_j}
\sum_{0 \leq i \leq m} \left(
  \left|x_{i+1} - x_i\right| \left|x_i - \tw
x_{i}\right| + \left|x_{i+1} - \tw x_{i+1}\right| \left|\tw
x_{i+1} - \tw x_{i}\right|\right) \label{eq:xs} \\
&\leq 4 \sum_{0 \leq i \leq m}
 \left(\max_{0 \leq k \leq m_j}
 \sup_{t^j_{k} + \zeta(e^j_{k}) < t < t^j_{k+1}, X_t \in \prt D}
|x^j_{k} - X_t| \right)^2 \nonumber \\
& \qquad +
 8 \left(\max_{0 \leq k \leq m_j}
 \sup_{t^j_{k} + \zeta(e^j_{k}) < t < t^j_{k+1}, X_t \in \prt D}
|x^j_{k} - X_t| \right)
\left(\sum_{0 \leq k \leq m_{j+1}}
|e_{t^{j+1}_k}(0)- e_{t^{j+1}_k}(\zeta-)|
\right) \nonumber \\
&= 4 ( m+1)
 \left(\max_{0 \leq k \leq m_j}
 \sup_{t^j_{k} + \zeta(e^j_{k}) < t < t^j_{k+1}, X_t \in \prt D}
|x^j_{k} - X_t| \right)^2 \nonumber \\
& \qquad +
 8 \left(\max_{0 \leq k \leq m_j}
 \sup_{t^j_{k} + \zeta(e^j_{k}) < t < t^j_{k+1}, X_t \in \prt D}
|x^j_{k} - X_t| \right)
\left(\sum_{0 \leq k \leq m_{j+1}}
|e_{t^{j+1}_k}(0)- e_{t^{j+1}_k}(\zeta-)|
\right). \nonumber
 \end{align}

We have the following estimate, similar to (\ref{eq:maxM}). For
any $\beta_5 < 2$, some $c_{19}<\infty$, and $j\geq j_1$,
\begin{align}
\bE\left (\max_{1\leq k \leq n_j} c_8 M^j_k \eps _j \right)^2
&\leq
\bE\left (\max_{1\leq k \leq n_j} c_8 \tw M^j_k \eps _j \right)^2
\nonumber \\
&\leq
\sum_{i \leq \beta_1 j} (c_8 i \eps_j)^2 + \sum_{i >  \beta_1 j}
(c_8 i \eps_j)^2
c_9 n_j (1/2)^i \nonumber \\
&\leq c_{17} \eps_j^2 |\log \eps_j|^3
+ c_{18} \eps_j^2 (\log \eps_j)^2
\leq c_{19} \eps_j^{\beta_5} .\label{eq:maxM2}
\end{align}
We now proceed as in (\ref{eq:X-M}). For any $\beta_5 < 2$ and
$j \geq j_1$,
\begin{align}
\bE&\left( \max_{0 \leq k \leq m_j}
\sup_{t^j_{k} + \zeta(e^j_{k}) < t < t^j_{k+1}, X_t \in \prt D}
|x^j_{k} - X_t| \right)^2 \nonumber \\
 &\leq \bE\left( \max_{0 \leq k \leq n_j}
\sup_{t^j_{k} + \zeta(e^j_{k}) < t < t^j_{k+1}, X_t \in \prt D}
|x^j_{k} - X_t| \right)^2
+ \rho_0^2 \bP(m_j \geq n_j) \nonumber \\
&\leq \bE\left( \max_{0 \leq k \leq n_j} c_8M^j_k \eps_j \right)^2
+ c_{20} \eps_j^2 \nonumber \\
&\leq  c_{19} \eps_j^{\beta_5} + c_{20} \eps_j^2 \leq c_{21}
\eps_j^{\beta_5}. \label{eq:X-M2}
\end{align}

Recall that $m$ is random and note that $m \leq m_{j+1}$. We
obtain the following from (\ref{eq:LDP1})
and (\ref{eq:X-M2}), for any $\beta_7 < 1$, by choosing
appropriate $\beta_5 < 2$ and $\beta_6 < -1$,
\begin{align}\label{eq:msup}
\bE&\left( (m+1)
 \left(\max_{0 \leq k \leq m_j}
 \sup_{t^j_{k} + \zeta(e^j_{k}) < t < t^j_{k+1}, X_t \in \prt D}
|x^j_{k} - X_t| \right)^2 \right) \\
&\leq
\bE\left( \eps_j^{\beta_6}
 \left(\max_{0 \leq k \leq m_j}
 \sup_{t^j_{k} + \zeta(e^j_{k}) < t < t^j_{k+1}, X_t \in \prt D}
|x^j_{k} - X_t| \right)^2 \right)
+ \rho_0^2 \bP(m+1 \geq \eps_j^{\beta_6}) \nonumber \\
&\leq c_{21} \eps_j^{\beta_6 + \beta_5} + c_{22} \eps_j^2
\leq c_{23} \eps_j^{\beta_7}. \nonumber
 \end{align}

Next we estimate the second term on the right hand side of
(\ref{eq:xs}) as follows. The number of excursions
$e_{t^{j+1}_k}$ with $|e_{t^{j+1}_k}(0)- e_{t^{j+1}_k}(\zeta-)| \in [\eps_{i+1}, \eps_i]$ is bounded by $m_{i+1}$, so
$$
\sum_{0 \leq k \leq m_{j+1}}
|e_{t^{j+1}_k}(0)- e_{t^{j+1}_k}(\zeta-)|
\leq
\sum_{i=j_1}^{j+1} m_i \eps_{i-1}.
$$
Hence, for any $\beta_9 <0$, we can choose
$\beta_8 < 0$, $\beta_1 < -1$ and $c_{23} < \infty$ so that for
all $j\geq j_1$,
\begin{align*}
& \left(\max_{0 \leq k \leq m_j}
 \sup_{t^j_{k} + \zeta(e^j_{k}) < t < t^j_{k+1}, X_t \in \prt D}
|x^j_{k} - X_t| \right) \left(\sum_{0 \leq k \leq m_{j+1}}
|e_{t^{j+1}_k}(0)- e_{t^{j+1}_k}(\zeta-)| \right) \\
& \leq
\left(\max_{0 \leq k \leq m_j}
 \sup_{t^j_{k} + \zeta(e^j_{k}) < t < t^j_{k+1}, X_t \in \prt D}
|x^j_{k} - X_t| \right) \sum_{i=j_1}^{j+1} m_i \eps_{i-1} \nonumber \\
& \leq
\left(\max_{0 \leq k \leq m_j}
 \sup_{t^j_{k} + \zeta(e^j_{k}) < t < t^j_{k+1}, X_t \in \prt D}
|x^j_{k} - X_t| \right) \sum_{i=j_1}^{j+1}  \eps_j^{\beta_8} n_i 2 \eps_i
+  \rho_0 \sum_{i=j_1}^{j+1} m_i \bone_{\{m_i\geq n_i\eps_j^{\beta_8}\}} 2 \eps_i  \nonumber \\
& \leq
\left(\max_{0 \leq k \leq m_j}
 \sup_{t^j_{k} + \zeta(e^j_{k}) < t < t^j_{k+1}, X_t \in \prt D}
|x^j_{k} - X_t| \right) 2 (j-j_1) \eps_j^{\beta_8 + \beta_1+1}
+  2 \rho_0 \sum_{i=j_1}^{j+1} m_i \bone_{\{m_i\geq n_i \eps_j^{\beta_8}\}} \eps_i  \nonumber \\
& \leq
c_{23} \eps_j^{\beta_9}
\left(\max_{0 \leq k \leq m_j}
 \sup_{t^j_{k} + \zeta(e^j_{k}) < t < t^j_{k+1}, X_t \in \prt D}
|x^j_{k} - X_t| \right)
+  2\rho_0 \sum_{i=j_1}^{j+1} m_i \bone_{\{m_i\geq n_i \eps_j^{\beta_8}\}} \eps_i . \nonumber
\end{align*}
This, (\ref{eq:X-M}) and (\ref{eq:LDP2}) imply that for any
$\beta_{10} < 1$, by choosing an appropriate $\beta_2<1$ and
$\beta_8,\beta_9 < 0$, we obtain for some $c_{26} < \infty$ and $j\geq j_1$,
\begin{align}\label{eq:xXd}
\bE& \left(\max_{0 \leq k \leq m_j}
 \sup_{t^j_{k} + \zeta(e^j_{k}) < t < t^j_{k+1}, X_t \in \prt D}
|x^j_{k} - X_t| \right) \left(\sum_{0 \leq k \leq m_{j+1}}
|e_{t^{j+1}_k}(0)- e_{t^{j+1}_k}(\zeta-)| \right) \\
& \leq
c_{23} \eps_j^{\beta_9}
\bE \left(\max_{0 \leq k \leq m_j}
 \sup_{t^j_{k} + \zeta(e^j_{k}) < t < t^j_{k+1}, X_t \in \prt D}
|x^j_{k} - X_t| \right)
+ 2\rho_0 \bE\left(\sum_{i=j_1}^{j+1} m_i \bone_{\{m_i\geq n_i \eps_j^{\beta_8}\}} \eps_i\right) ,
\nonumber \\
&\leq
c_{24} \eps_j^{\beta_9} \eps_j^{\beta_2}
+ c_{25} \sum_{i=j_1}^{j+1} \eps_j^2 \eps_i
\leq c_{26} \eps_j^{\beta_{10}}. \nonumber
\end{align}

We combine (\ref{eq:xs}), (\ref{eq:msup}) and (\ref{eq:xXd}) to
see that for any $\beta_{10} < 1$, some $c_{27}< \infty$ and
all $j\geq j_1$,
\begin{align*}
 \bE \left(\sup_{\eps_{j+1} \leq \eps' < \eps_j}
 \sum_{0 \leq i \leq m} \left(
  \left|x_{i+1} - x_i\right| \left|x_i - \tw
x_{i}\right| + \left|x_{i+1} - \tw x_{i+1}\right| \left|\tw
x_{i+1} - \tw x_{i}\right|\right) \right)
&\leq c_{27} \eps_j^{\beta_{10}}.
\end{align*}
We use this estimate and (\ref{eq:2ndterm}) to see that for any
$\beta_{10} < 1$, some $c_{27}< \infty$ and all $j\geq j_1$,
\begin{align}\label{eq:2ndterm1}
\bE &\left( \sup_{\eps_{j+1} \leq \eps' < \eps_j}
\sum_{i=0}^m
 \left\|
\pi_{i+1} \circ
\left(
\pi_i - \tw\pi_{i+1}
\right)
\circ \tw \pi_i
\circ e^{\Delta_i\tw \sh _i}
\right\| \right) \\
& \leq
 \bE\left( \sup_{\eps_{j+1} \leq \eps' < \eps_j}
 \sum_{i=0}^m c_{4}
\left(
  \left|x_{i+1} - x_i\right| \left|x_i - \tw
x_{i}\right| + \left|x_{i+1} - \tw x_{i+1}\right| \left|\tw
x_{i+1} - \tw x_{i}\right|\right)
  \right) \leq c_{27} \eps_j^{\beta_{10}}. \nonumber
\end{align}

It follows (\ref{eq:vminv}), (\ref{eq:decom}), (\ref{eq:1stterm1}),
and (\ref{eq:2ndterm1}) that for any $\beta_{10} < 1$,
some $c_{28}< \infty$ and all $j\geq j_1$,
\begin{equation*}
\bE \left( \sup_{\eps_{j+1} \leq \eps' < \eps_j}
 |\bv^j - \bv '| \right) \leq c_{28} \eps_j^{\beta_{10}} |\bv_0|
 = c_{28} 2^{-\beta_{10} j} |\bv_0|.
\end{equation*}
This implies that $ \sum _{j\geq j_1} \bE \left(
\sup_{\eps_{j+1} \leq \eps' < \eps_j} |\bv^j - \bv'| \right) <
\infty $, and, therefore, a.s.,
\begin{equation}\label{eq:vcon}
 \sum _{j\geq j_1}  \left( \sup_{\eps_{j+1} \leq \eps'
< \eps_j} |\bv^j - \bv'| \right) < \infty.
\end{equation}
We extend the notation $\bv'$ from $\eps'$ in the range
$[\eps_{j+1}, \eps_j)$ to all $\eps'>0$, in the obvious way. It
is elementary to see that (\ref{eq:vcon}) implies that $\bv_1
:= \lim_{\eps' \downarrow 0} \bv'$ exists. For every $\eps'>0$,
the mapping $\bv_0 \to \bv'$ is linear, so the same can be said
about the mapping $\bv_0 \to \bv_1 := \A_1 \bv_0$.

Note that the right hand side of (\ref{eq:vminv}) corresponding to $r\in[0,1)$ is less than or equal
to the right hand side of (\ref{eq:vminv}) in the case $r=1$.
Hence, we can strengthen (\ref{eq:vcon}) to the claim that a.s.,
\begin{equation*}
 \sum _{j\geq j_1}  \left( \sup_{0\leq r \leq 1} \sup_{\eps_{j+1} \leq \eps'
< \eps_j} |\bv_r^j - \bv_r'| \right) < \infty,
\end{equation*}
where $\bv_r^j$ and $\bv_r'$ are defined in a way analogous to
$\bv^j$ and $\bv'$, relative to $r\in[0,1]$. The analogous argument shows that for any integer $r_0>0$,
a.s.,
\begin{equation*}
 \sum _{j\geq j_1}  \left( \sup_{0\leq r \leq r_0} \sup_{\eps_{j+1} \leq \eps'
< \eps_j} |\bv_r^j - \bv_r'| \right) < \infty.
\end{equation*}
We use the same argument as above to conclude that for any
$\bv_0$, with probability 1, $\A_{r,\eps}\bv_0\to \A_r \bv_0$
uniformly on compact sets.

It remains to show that $\A_r$ has rank $n-1$. Without loss of
generality, we will consider only $r=1$. Recall definition
(\ref{def:vr}) of $\bv_r$ and note that $\pi_{x_0} \bv_0 \in
\tngt_{x_0} \prt D$. It will suffice to show that for any
$\bw\in \tngt_{x_0} D$ such that $\bw \ne 0$, we have $\A_1 \bw
\ne 0$.

Recall the definition of $x^j_k$'s and related notation from
the beginning of the proof. Recall from (\ref{eq:e^S-estlow})
that for some $c_{29}< \infty$ depending only on $D$, all $x\in \prt
D$, $\bz \in \R^n$, and all $t\geq 0$, we have $|e^{t \sh (x)}
\bz| \ge e^{-c_{29} t}|\bz|$. Therefore, for any $\bw\in
\tngt_{x_0} D$,
\begin{align*}
|\bv^j| &=
|
\exp(\Delta\ell^j_{m_j} \sh(x^j_{m_j})) \pi_{x^j_{m_j}}
\cdots
\exp(\Delta\ell^j_1 \sh(x^j_1)) \pi_{x^j_1}
\exp(\Delta \ell^j_0 \sh(x^j_0)) \pi_{x^j_0} \bw| \\
&\geq \exp\left(-c_{29}\sum_{i=0}^{m_j} \Delta \ell_i \right)
 |\pi_{x^j_{m_j}} \pi_{x^j_{m_j-1}} \cdots \pi_{x^j_1} \pi_{x^j_0} \bw|  \\
& = c_{30} |\pi_{x^j_{m_j}} \pi_{x^j_{m_j-1}} \cdots \pi_{x^j_1} \pi_{x^j_0} \bw|.
\end{align*}
It follows that
\begin{align*}
\frac{|\bv^j|}{|\bw|} =
\prod _{k=1}^{m_j}
\frac{|\pi_{x^j_k}  \cdots \pi_{x^j_1} \pi_{x^j_0} \bw|}
{| \pi_{x^j_{k-1}} \cdots \pi_{x^j_1} \pi_{x^j_0} \bw|},
\end{align*}
and, therefore,
\begin{align*}
\log|\bv^j| = \log |\bw| +
\sum _{k=1}^{m_j}
\left( \log |\pi_{x^j_k}  \cdots \pi_{x^j_2} \pi_{x^j_1} \bw|
- \log | \pi_{x^j_{k-1}}  \cdots \pi_{x^j_2} \pi_{x^j_1} \bw| \right).
\end{align*}

By the Pythagorean theorem, $|\bz|^2 = |\pi_x \bz|^2 + \<\bz/|\bz|, \n(x)\>^2 |\bz|^2$.
This implies that for some $c_{31}<\infty$, if $\bz \in \tngt_y
\prt D$ then
 \begin{equation*}
 |\pi_x \bz| \geq \left(1-c_{31} |x-y|^2 \right) |\bz|.
 \end{equation*}
Thus we can find $\rho_1 >0 $ so small that for some $c_{32}$ and all $|x-y| \leq \rho_1$ and
$\bz \in \tngt_y \prt D$,
 \begin{equation*}
 \log |\pi_x \bz| \geq  \log |\bz| - c_{32} |x-y|^2.
 \end{equation*}
Therefore,
 \begin{align}\label{eq:logest}
\log|\bv^j| &\geq \log |\bw| -
c_{32}\sum _{k=1}^{m_j} |x^j_k - x^j_{k+1}|^2
\bone_{\{|x^j_k - x^j_{k+1}| \leq \rho_1\}}\\
&\qquad +
\sum _{k=1}^{m_j}
\left( \bone_{\{|x^j_k - x^j_{k+1}| > \rho_1\}}
\log \frac{|\pi_{x^j_k}  \cdots \pi_{x^j_1} \pi_{x^j_0} \bw|}
{| \pi_{x^j_{k-1}} \cdots \pi_{x^j_1} \pi_{x^j_0} \bw|}
 \right). \nonumber
 \end{align}
We make $\rho_1$ smaller, if necessary, so that $\rho_1/2 =
\eps_{j_2}$ for some integer $j_2$. Note that the set of
excursions $e_{t^{j_2}_k}$ is finite, with cardinality
$m_{j_2}$.

The hitting distribution of $\prt D$ for any excursion law $H^x$ is
absolutely continuous with respect to the surface area measure
on $\prt D$, because the same is true for Brownian motion.
This, (\ref{old4.1}) and Assumption \ref{a:A1} imply that with
probability 1, for all $k=1,2,\dots, m_{j_2}$, we have
$|\<\n(e_{t^{j_2}_k}(0)), \n(e_{t^{j_2}_k}(\zeta-))\>| > \delta $, for some random $\delta >0$. For large $j$,
because of continuity of reflected Brownian motion paths, and
because excursions are dense in the trajectory, the only points
$x^j_{k+1}$ such that $|x^j_k - x^j_{k+1}| > \rho_1$ can be the
endpoints of excursions $e_{t^{j_2}_i}$, $i=1,2,\dots,
m_{j_2}$.

Fix a point $e_{t^{j_2}_i}$ and let $k(j)$ be such that
$x^j_{k(j)} = e_{t^{j_2}_i}$. Then $x^j_{k(j)-1}\to x^j_{k(j)}$
as $j \to \infty$, again by the continuity of reflected
Brownian motion, and because excursions are dense in the
trajectory. It follows that for large $j$, for all pairs
$(x^j_k, x^j_{k+1})$ with $|x^j_k - x^j_{k+1}|
> \rho_1$, we have $|\<\n(x^j_k), \n(x^j_{k+1})\>|
> \delta/2 $. This implies that, a.s., for some random
$U>-\infty$, and all sufficiently large $j$,
\begin{equation}\label{eq:largeexc}
\sum _{k=1}^{m_j}
\left( \bone_{\{|x^j_k - x^j_{k+1}| > \rho_1\}}
\log \frac{|\pi_{x^j_k}  \cdots \pi_{x^j_1} \pi_{x^j_0} \bw|}
{| \pi_{x^j_{k-1}} \cdots \pi_{x^j_1} \pi_{x^j_0} \bw|}
 \right) > U.
\end{equation}

In view of (\ref{eq:diam}) and (\ref{eq:msup}), for any
$\beta_7 <1$,
\begin{align}\label{eq:exc}
\bE &\left( \sum_{i=0}^{m_j} |x^j_i - x^j_{i+1}|^2 \right) \\
&\leq 8 \bE \left( m_j \left( \max_{0 \leq k \leq m_j}
\sup_{t^j_{k} + \zeta(e^j_{k}) < t < t^j_{k+1}, X_t \in \prt D} |x^j_{k} - X_t|
\right)^2 \right)
+ 8\bE \left( \sum_{k=1}^{m_j} |e_{t^j_k}(0)-e_{t^j_k}(\zeta-)|^2
\right) \nonumber \\
&\leq c_{23} \eps_j^{\beta_7}
+ 8\bE \left( \sum_{k=1}^{m_j} |e_{t^j_k}(0)-e_{t^j_k}(\zeta-)|^2
\right). \nonumber
\end{align}
By (\ref{old4.1}) and (\ref{eq:H1}), the expected number of
excursions $e_s$ with $|e_s(\zeta-) - e_s(0)| \in [2^{-i-1},
2^{-i}]$ and $s\in[0,1]$ is bounded by $c_{33} 2^i$. It follows
that for some $c_{34}<\infty$, not depending on $j$,
\begin{equation*}
\bE \left( \sum_{k=1}^{m_j} |e_{t^j_k}(0)-e_{t^j_k}(\zeta-)|^2
\right) \leq \sum_{i=j_1}^j c_{34} 2^{-2i} 2^i < c_{35} < \infty,
\end{equation*}
and this combined with (\ref{eq:exc}) yields
\begin{equation*}
\sup_{j\geq j_1} \bE \left( \sum_{i=0}^{m_j} |x^j_i - x^j_{i+1}|^2 \right)
< \infty.
\end{equation*}
In view of (\ref{eq:logest}) and (\ref{eq:largeexc}),
 \begin{equation*}
\liminf_{j\to \infty} \bE (\log|\bv^j| -U) \geq \log |\bw|
- \limsup_{j\to \infty} \bE \left(
c_{32}\sum _{k=1}^{m_j} |x^j_k - x^j_{k+1}|^2\right) > -\infty,
 \end{equation*}
so, with probability 1, $\liminf_{j\to \infty} |\bv^j|
>0$, and, therefore, $|\bv_1| \ne 0$.
\end{proof}

\end{document}